\documentclass[draft,11pt,english]{smfart}
\usepackage[utf8]{inputenc}
\usepackage[T1]{fontenc}
\usepackage[francais,english]{babel}
\usepackage{amssymb,url,xspace,smfthm}
\usepackage{lineno,hyperref}
\usepackage{bm}
\author[C.-M.~Marle]{Charles-Michel Marle}
\address{Honorary Professor, retired from \hbox{Universit\'e Pierre et Marie Curie} (today Sorbonne Universit\'e)}
 
\address{Home address: 27, avenue du 11 novembre 1918 \\92190 Meudon\\France}

\email{cmm1934@orange.fr}

\email{charles-michel.marle@math.cnrs.fr}

\urladdr{http://marle.perso.math.cnrs.fr/}

\title[Gibbs states for systems with symmetries]{On Gibbs states of mechanical systems with symmetries}
\NumberTheoremsAs{subsubsection}
\SwapTheoremNumbers
\begin{document}
\newenvironment{defis}{\begin{enonce}[definition]{Definitions}}{\end{enonce}}
\newenvironment{remas}{\begin{enonce}[remark]{Remarks}}{\end{enonce}}

\frontmatter
\begin{abstract}
Gibbs states for the Hamiltonian action of a Lie group on a symplectic manifold 
were studied, and their possible applications in Physics and Cosmology were 
considered, by the French mathematician and physicist Jean-Marie Souriau. They 
are presented here with detailed proofs of all the stated results. Using an 
adaptation of the cross product for pseudo-Euclidean three-dimensional vector 
spaces, we present several examples of such Gibbs states, together with the 
associated thermodynamic functions, for various two-dimensional symplectic 
manifolds, including the pseudo-spheres, the Poincaré disk 
and the Poincaré half-plane. 
\end{abstract}

\begin{altabstract}
Les états de Gibbs sur une variété symplectique associés à l'action 
hamiltonienne d'un groupe de Lie sur cette variéré ont été étudiés par le 
mathématicien et physicien Jean-Marie Souriau, qui en a aussi considéré des 
applications en Physique et en Cosmologie. Ils sont décrits ici avec la preuve 
détaillée de tous les résultats présentés. Grâce à une adaptation du produit 
vectoriel aux espaces vectoriels pseudo-euclidiens de dimension $3$, plusieurs 
exemples de tels états de Gibbs sont déterminés, ainsi que les fonctions
thermodynamiques qui leur sont associées, pour diverses variétés symplectiques 
de dimension 2, notamment les pseudo-sphères, le disque de Poincaré et 
le demi-plan de Poincaré.
\end{altabstract}


\altkeywords{Variétés symplectiques et de Poisson, mesure de Liouville, systèmes hamiltoniens, 
applications moment, états de Gibbs, équilibre thermodynamique, opérateur étoite de Hodge, 
transformations de Möbius, disque de Poincaré, demi-plan de Poincaré}
\thanks{This work was not supported by any public or private subvention other than the author's property}

\dedicatory{In memory of  the French mathematician and physicist Jean-Marie Souriau (1922--2012)}
\maketitle

%Personal macros
%
\newcommand{\field}[1]{\mathbb{#1}}
\newcommand{\CC}{\field{C}}
\newcommand{\DD}{\field{D}}
\newcommand{\FF}{\field{F}}
\newcommand{\HH}{\field{H}}
\newcommand{\KK}{\field{K}}
\newcommand{\NN}{\field{N}}
\newcommand{\PP}{\field{P}}
\newcommand{\QQ}{\field{Q}}
\newcommand{\RR}{\field{R}}
\newcommand{\ZZ}{\field{Z}}
\newcommand{\gras}[1]{\bf{#1}}
\newcommand{\F}{\gras{F}}

\def\id{\mathop{\rm id}\nolimits}
\def\fdet{\mathop{\hbox{d\'et}}\nolimits}
\def\Ad{\mathop{\rm Ad}\nolimits}
\def\ad{\mathop{\rm ad}\nolimits}
\def\orth{\mathop{\rm orth}\nolimits}
\def\toto{\mathop{\rightrightarrows}\limits}
\def\ch{\mathop{\rm cosh}\nolimits}
\def\sh{\mathop{\rm sinh}\nolimits}
\def\th{\mathop{\rm tanh}\nolimits}
\def\coth{\mathop{\rm coth}\nolimits}
\def\div{\mathop{\rm div}\nolimits}
\def\rot{\mathop{\rm rot}\nolimits}
\def\rang{\mathop{\rm rang}\nolimits}
\def\vect#1{\overrightarrow{\mathstrut#1}}
\def\d{\mathrm{d}}
\def\mathi{\mathrm{i}}
\def\scal{\mathop{\rm scal}\nolimits}
\def\GL{\mathop{\mathrm{GL}}\nolimits}
\def\gl{\mathop{\mathfrak{gl}}\nolimits}
\def\SO{\mathop{\mathrm{SO}}\nolimits}
\def\SL{\mathop{\mathrm{SL}}\nolimits}
\def\O{\mathop{\mathrm O}\nolimits}
\def\SU{\mathop{\mathrm{SU}}\nolimits}
\def\U{\mathop{\mathrm U}\nolimits}
\def\so{\mathop{\mathfrak{so}}\nolimits}
\def\su{\mathop{\mathfrak{su}}\nolimits}
\def\im{\mathop{\mathrm{im}}\nolimits}
\def\Sp{\mathop{\mathrm{Sp}}\nolimits}
\def\sp{\mathop{\mathfrak{sp}}\nolimits}
\def\sl{\mathop{\mathfrak{sl}}\nolimits}
\def\lin{\mathop{\mathrm{lin}}\nolimits}
\def\Mot{\mathop{\mathrm{Mot}}\nolimits}
\def\date{\mathop{\mathrm{date}}\nolimits}
\setcounter{tocdepth}{2} 
\tableofcontents

\mainmatter

% \linenumbers

\section{Introduction}
\label{sec:1}
The French mathematician and physicist Jean-Marie~Souriau (1922--2012) considered, 
first in \cite{Souriau1966}, then in his book 
\cite{Souriau1969}, Gibbs states on a symplectic manifold built with the moment map
of the Hamiltonian action of a Lie group, and the associated thermodynamic functions. 
In several later papers \cite{Souriau1974, Souriau1975, Souriau1984}, he developed these
concepts and considered their possible applications in Physics and in Cosmology.
A partial translation in English of these three papers, made by Frédéric Barbaresco,
is available at \cite{Barbaresco2020b}.
\par\smallskip

Recently, under the name \emph{Souriau's Lie groups thermodynamics}, these
Gibbs states and the associated thermodynamic functions were considered by 
several scientists, notably by Frédéric Barbaresco, for their possible 
applications in today very fashionable scientific topics, such as geometric 
information theory, deep learning and  machine learning  
\cite{Barbaresco2014, Barbaresco2015, Barbaresco2016, Barbaresco2020, 
BarbarescoGayBalmaz2020, deSaxce2015, deSaxce2016}. Although including these 
topics in a reasearch program seems to be, nowadays, a good way to obtain a
public funding, I am not going to speak about them, since they are far from
my field of knowledge. I will rather stay on Gibbs states and their 
possible applications in classical and relativistic Mechanics.
\par\smallskip

Long before the works of Souriau, Gibbs states associated to a Hamiltonian Lie 
group action were considered by  the American scientist Josiah Willard Gibbs 
(1839--1903). In his book \cite{Gibbs1881} published in 1902, he clearly 
described Gibbs states in which the components of the total angular momentum
(which are the components of the moment map of the action of the group of rotations
on the phase space of the considered system) appear, on the same footing as the 
Hamiltonian. He even considered  Gibbs states involving conserved quantities
more general than those associated with the Hamiltonian action of a Lie group. 
In this domain, Souriau's main merits do not lie, in my opinion, in the 
consideration of Gibbs states for the Hamiltonian action of a Lie group, 
a not so new idea, but rather in the use of the \emph{manifold of motions} of 
a Hamiltonian system instead of the use of its \emph{phase space},  and his 
introduction, under the name of \emph{Maxwell's principle}, 
of the idea that a symplectic structure should exist on the manifold of motions 
of systems encountered as well in classical Mechanics as in relativistic Physics. 
He therefore considered Gibbs states for Hamiltonian actions, on a 
symplectic manifold, of various Lie groups, including the Poincaré group, often 
considered in Physics as a group of symmetries for isolated relativistic systems.
He was well aware of the fact that Gibbs states for the Hamiltonian action
of the full considered groups may not exsist, which led him to carefully discuss
the physical meaning and the possible applications of Gibbs states associated 
to the action of some of their subgroups.

\par\smallskip

Section~\ref{sec:2} begins with a reminder about some concepts used in statistical 
mechanics, notably the concepts of \emph{statistical states} and of \emph{entropy},
and about the use of Hamiltonian vector fields in Mechanics.    
Gibbs states in the special case in which the only conserved quantity considered 
is the Hamiltonian, and the associated thermodynamic functions, are then briefly discussed. 
Their physial interpretation as states of thermodynamic equilibrium is discussed. The
relation of the real parameter $\beta$ used to index statistical states with 
the temperature is explained.
\par\smallskip

The notion of \emph{manifold of motions} of a Hamiltonian dynamical system is
presented in Section~\ref{sec:3}. Then Gibbs states for the Hamiltonian action of a Lie group
on a symplectic manifold are discussed, with full proofs of all the stated results. 
Most of these proofs can be found in Souriau's book \cite{Souriau1969}, which some readers 
may find difficult to access. A good English translation of this book is available,
which faithfully preserves the language and the notations of the author. 
\par\smallskip

In section~\ref{sec:4}, some examples of Gibbs states
are presented, together with the associated  thermodynamic functions. 
The main tool used in this section is an adaptation of the well known \emph{cross product}
for three-dimensional oriented, pseudo-Euclidean vector spaces. Remarkably, according to
\cite{CrossProductWikipedia}, the cross product of two elements of a three-dimensional, 
oriented, Euclidean vector space appeared for the first time in the lecture notes 
\emph{Elements of Vector Analysis} \cite{Gibbs1881}, privately written in 1881 
for students in physics by Gibbs, one of the most important founders of 
statistical mechanics.
\par\smallskip

The readers will find at the beginning of sections \ref{sec:3} and \ref{sec:4}
a more detailed presentation of the contents of these sections.

\section{Some concepts used in statistical mechanics}
\label{sec:2}

\subsection{The birth of statistical mechanics}
\label{subssec:2.1}

In his book \emph{Hydrodynamica} published in 1738, Daniel Bernoulli (1700--1782)
considered fluids (gases as well as liquids) as made of a very large number of moving
particles. He explained that the pressure in the fluid is the result of collisions of 
the moving particles against the walls of the vessel in which it is contained, or against
the probe which measures the pressure. 
\par\smallskip

Daniel Bernoulli's idea remained ignored by
most scientists for more than one hundred years. It is only in the second half of 
the XIX-th century that some scientists, notably Rudolf Clausius (1822--1888), 
James Clerk Maxwell (1831--1879) and Ludwig Eduardo Boltzmann(1844--1906), considered 
Bernoulli's idea as reasonable. As soon as 1857, Clausius began the elaboration 
of a \emph{kinetic theory of gases} aiming at the explanation of macroscopic properties
of gases (such as temperature, pressure and other thermodynamic properties), 
starting from the equations which govern the motions of the moving particles. 
Around 1860, Maxwell determined the probability distribution of the moving particles 
velocities in a gas in thermodynamic equilibrium. For a gas not in thermodynamic 
equilibrium, an evolution equation for this probability distribution was obtained 
by Boltzmann in 1872. Using probabilistic arguments about the way in which 
collisions of particles can occur, Boltzmann introduced a quantity, denoted by 
$H$\footnote{In Boltzmann's mind, this letter was probably the Greek boldface 
letter \emph{\^Eta} rather than the Latin letter H.} which, as a functon of time, 
always montonically decrases. Boltzmann's H function is now identified with the 
opposite of the \emph{entropy} of the gas. On this basis, Josiah Willard Gibbs 
(1839--1903) laid the foundations of a new branch of theoretical physics, which 
he called \emph{statistical mechanics} \cite{Gibbs1902}.
\par\smallskip

In the first half of the XX-th century, scientists understood that the motions of 
molecules in a material body do not perfectly obey Newton's laws of classical 
mechanics, and that the laws of quantum mechanics should be used instead. The 
basic concepts of statistical mechanics established by Gibbs were general enough 
to remain valid in this new framework, and to be used for liquids or solids as 
well as for gases.

\subsection{Statistical states and entropy}
\label{subsec:2.2} 
In this subsection, after a reminder of some well known facts about the use of 
Hamiltonian vector fields in classical mechanics and about symplectic manifolds, 
the important concept of a \emph{statistical state} is presented
and the definition of its \emph{entropy} is given.

\begin{enonce}[definition]
{The use of Hamiltonian vector fields in classical mechanics}
\label{HamiltonianVectorFieldsInMechanics}
  
Let us recall how the evolution with time of the state of a material body is 
mathematically described by a dynamical system, in the framework of classical mechanics. 
The physical time $\mathcal T$ is a one-dimensional real, oriented  affine space, identified
with $\RR$ once a unit and an origin of time are chosen. The set of all 
possible kinematic states of the body is a symplectic manifold $(M,\omega)$, very 
often a cotangent bundle, traditionnaly called the \emph{phase space} of the 
system. For an isolated system, a smooth real-valued function $H$, 
defined on $M$, called a \emph{Hamiltonian} for the system, determines all its 
possible evolutions with time. Let indeed $X_H$ be the unique smooth vector field, 
defined on $M$, which satisfies the equality
 $$\mathi(X_H)\omega=-\d H\,.\eqno(*)
 $$
It is called the \emph{Hamiltonian vector field} admitting the function $H$ as a Hamiltonian.
Any possible evolution with time of the system is described by a smooth curve 
$t\mapsto \varphi(t)$, defined on an open interval in $\RR$, with values in $M$, which 
is a maximal integral curve of the differential equation, called 
\emph{Hamilton's equation}, in honour of  the Irish mathematician William Rowan Hamilton
(1805--1865), 
 $$\frac{\d \varphi(t)}{\d t}= X_H\bigl(\varphi(t)\bigr)\,.\eqno(**)
 $$
The Hamiltonian $H$ is a \emph{first integral} of this differential equation~: it means that 
for each smooth curve $t\mapsto \varphi(t)$, solution of this differential equation, 
$H\bigl(\varphi(t)\bigr)$ is a constant.
\par\smallskip
    
More generally, when the system is not isolated, its Hamiltonian $H$ is a smooth function
defined on $\RR\times M$ (or on an open subset of $\RR\times M$) since it may depend
on time. The Hamiltonian vector field $X_H$ which admits such a function as Hamiltonian
is still, for each time $t\in \RR$,  determined by equation $(*)$ above, in the righ-hand side
of which the differential $\d H$ must be calculated, for each $t\in\RR$, as its \emph{partial differential}
with respect to the variable $x\in M$, the time $t\in \RR$ being considered as fixed.
Therefore $X_H$ is a \emph{time-dependent vector field} on $M$, \emph{i.e.}, a smooth
map, defined on some open subset of $\RR\times M$, with values in the tangent bundle $TM$, such that for
each fixed $t\in \RR$, the map $x\mapsto X_H(t,x)$ is an usual smooth vector field defined
on some open subset of $M$. Any possible evolution with time of the system is still described by 
a smooth curve $t\mapsto \varphi(t)$, which is a maximal integral curve of the differential equation $(**)$ above, 
wich now must be written as  
 $$\frac{\d \varphi(t)}{\d t}= X_H\bigl(t,\varphi(t)\bigr)\,,\eqno({*}{*}{*})
 $$ 
in order to indicate that $X_H$ may depend on $t\in\RR$ as well as on 
$\varphi(t)\in M$. In this case the Hamiltonian $H$ is no more a first integral 
of this differential equation.
\end{enonce}

\begin{enonce}[definition]{The Liouville measure on a symplectic manifold}
\label{LiouvilleMeasure}
Let $(M,\omega)$ be a $2n$-dimensional symplectic manifold. Let $(U,\varphi)$ be
an admissible chart of $M$. For each $x\in M$, we set
 $$\varphi(x)=(q^1,\ldots,q^n,\allowbreak p_1,\ldots,p_n)\in\varphi(U)\subset \RR^{2n}\,.
 $$
The chart $(U,\varphi)$ is said to be \emph{canonical}, or to be a
\emph{Darboux chart}, if the local expression of $\omega$ in $U$ is
 $$\omega=\sum_{i=1}^n\d p_i\wedge\d q^i\,.
 $$
The local coordinates $q^1,\ldots,q^n,p_1,\ldots,p_n$ in this chart are called 
\emph{canonical coordinates} or \emph{Darboux coordinates}. The famous 
\emph{Darboux theorem}, so named in honour of the French mathematician
Gaston Darboux (1842--1917), asserts that any point in $M$ is an element of the domain of a
canonical chart. By using this theorem, one can prove the existence of a unique positive measure
on the Borel $\sigma$-algebra\footnote{The $\sigma$-algebra of a topological space $M$ 
is the smallest family of subsets of $M$ which contains all open subsets and is stable
by complementation and by intersections of countable subfamilies. It is so named 
in honour of the French mathematician \'Emile Borel (1871--1956).}
of $M$, called the \emph{Liouville measure}, in hounour of the French mathematician
Joeph Liouville (1809--1882) and denoted by $\lambda_\omega$,
such that for any measurable subset $A$ of $M$ contained in the domain $U$ of a canonical chart
$(U,\varphi)$ of $M$, such that $\varphi(A)$ is a bounded subset of $\RR^{2n}$,
 $$\lambda_\omega(A)=\int_{\varphi(A)}\d q^1\ldots\d q^n\d p_1\ldots\d p_n\,.
 $$ 
The Liouville measure is invariant by symplectomorphisms, which means that
its  direct image $\Phi_*\lambda_\omega$ by any symplectomorphism $\Phi:M\to M$
is equal to $\lambda_\omega$.
\end{enonce}

\begin{defis}
\label{DefiStatisticalState} 
Let $(M,\omega)$ be a symplectic manifold and $\lambda_\omega$  its Liouville measure.

\par\smallskip\noindent
{\rm 1.}\quad A \emph{statistical state} on $M$ is a probability measure 
$\mu$ on the Borel $\sigma$-algebra of $M$. The statistical state $\mu$ is said 
to be \emph{continuous} (respectively, \emph{smooth}) when it can be written as
$\mu=\rho\lambda_\omega$, where $\rho$ is a continuous function (respectively, a 
smooth function) defined on $M$. The function $\rho$ is then said to be the 
\emph{probability density} (or simply the \emph{density})
of the statistical state $\mu$ with respect to $\lambda_\omega$.

\par\smallskip\noindent
{\rm 2.}\quad Let $\mu$ be a statistical state on $M$ and $f$ be a function, 
defined on $M$, which takes its values in $\RR$ or in a 
finite-dimensional vector space.
When $f$ is integrable on $M$ with respect to the measure $\mu$, its integral is called
the \emph{mean value} of $f$ in the statistical state $\mu$, and denoted by
${\mathcal E}_\mu(f)$. When the statistical state $\mu$ is continuous,
with the continuous function $\rho$ as probability density with respect to $\lambda_\omega$,
the mean value of $f$ in the statistical state $\mu$ is, by a slight abuse of notations,
denoted by ${\mathcal E}_\rho(f)$. Its expression is
 $${\mathcal E}_\rho(f)= \int_Mf(x)\rho(x)\,\lambda_\omega(\d x)\,.
 $$
\end{defis}

\begin{enonce}[definition]{Comments about the use of statistical states}
\label{CommentsStatisticalStates}
When the considered dynamical system, determined by the Hamiltonian vector field
$X_H$, is made of a large number $N$ of moving particles, the dimension of the 
symplectic manifold $(M,\omega)$ which represents 
the set of all its possible kinematic states is very large~: at least $6N$, 
and even more when the particles are not treated as material points. A perfect 
knowledge of each element of $M$ is not possible, which explains the use of 
statistical states in classical Mechanics. In this framework, when the state of the
considered system at a given time $t_0$ is mathematically described by a 
statistical state $\mu$, it means that instead of looking at the evolution in time
of a unique system whose kinematical state at time $t_0$ is a given element
$x_0\in M$, one is going to look at the evolution in time of a whole family of
similar systems. The evolution with time of each of these systems is described by 
the differential equation determined by $X_H$, and its kinematical states at time 
$t_0$ can be any point in the support\footnote{The support of a measure $\mu$ 
defined on the Borel $\sigma$-algebra of a topological state $M$ is the 
closed subset of $M$, complementary to the open subset made by points contained 
in an open subset $U$ of $M$ such that $\mu(U)=0$.} of $\mu$.
\par\smallskip

When, instead of classical mechanics, quantum mechanics is used for the 
mathematical description of the evolution with time of the state of a physical 
system, the use of statistical states is not due to an imperfect knowledge of the 
initial state of the system~: it is mandatory. Informations about the evolution 
with time of the state of a system given by quantum mechanics are indeed
always probabilistic. By nature, quantum mechanics is always statistical.
\end{enonce}

\begin{enonce}[definition]{Examples\phantom{bla bla azertyuiopqsdfghjklmwxcvbn}}
\label{ExamplesStatisticalStates}

\noindent
{\rm 1.}\quad Let $x_1\,,\, x_2\,,\ldots,\,x_N$ be $N$
pairwise distinct points in $M$, and $k_1\,,\, k_2\,,\ldots,\,k_N$ be real numbers
satisfying $k_i>0$ for all $i\in\{1,\ldots,N\}$ and $\sum_{i=1}^N k_i=1$.
For each $i\in\{1,\ldots,N\}$, let $\delta_{x_i}$ be the Dirac measure at $x_i$, 
whose value $\delta_{x_i}(A)$ for a measurable subset $A$ of $M$ is $0$ when 
$x_i\notin A$ and $1$ when $x_i\in A$. The measure 
 $\mu=\sum_{i=1}^N k_i\delta_{x_i}
 $ 
is a statistical state, which is neither continuous, nor smooth. The mean value 
of a function $f$ in the statistical state $\mu$ is $\sum_{i=1}^Nk_if(x_i)$.
\par\smallskip

For each $i\in\{1,\ldots,N\}$, the measure $\delta_{x_i}$ is a statistical state 
in which the kinematical state of the system is the point $x_i$, with a 
probability $1$. One can say that $\delta_{x_i}$ is a state in the usual sense. 
In the statitical state $\mu$, the kinematical state of the system is a random 
variable which can take each value $x_i$ with the probability $k_i$.
 
\par\smallskip\noindent
{\rm 2.}\quad Still under the same assumptions, for each $i\in\{1,\ldots,N\}$,
let $U_i$ be a neighbourhood of $x_i$ and $\varphi_i$ be a positive valued, smooth 
function, with compact support contained in $U_i$, satisfying the equality
$\int_Mf(x)\lambda_\omega(\d x)=1$. The measure $\nu$ whose probability density with respect 
to the Liouville measure $\lambda_\omega$ is $\rho_\nu=\sum_{i=1}^N k_i\varphi_i$
is a smooth statistical state, which can be considered as a smooth approximation of
the discrete statistical state $\mu$ considered above. Such smooth approximations
of non-smooth statistical states were extensively used by the founder of geostatistics,
the French mathematician and geologist Georges Matheron (1930--2000) \cite{Matheron1967} .   
\end{enonce}

\begin{rema}
Let $\mu$ be a continuous statistical state on the symplectic manifold $(M,\omega)$
and $\rho$ its probability density with respect to the Liouville measure
$\lambda_\omega$. For each measurable subset $A$ of $M$, we have
 $$\mu(A)=\int_A\rho(x)\,\lambda_\omega(\d x)\,,\quad\hbox{so for $A=M$,}\quad 
    \mu(M)=\int_M\rho(x)\,\lambda_\omega(\d x)=1\,. 
 $$
The function $\rho$ therefore takes its values in $\RR^+$ and is integrable 
on $M$ with respect to the Liouville measure.
\end{rema}

\begin{enonce}[definition]{Evolution with time of a statistical state}
\label{EvolutionStatisticalState}
Let $(M,\omega)$ be a symplectic manifold, $H\in C^\infty(M,\RR)$ be a smooth  
Hamiltonian on $M$ which does not depend on time and $X_H$ be the 
associated Hamiltonian vector field on $M$. We denote by $\Phi^{X_H}$ the reduced
flow\footnote{The \emph{full flow}, or in short the \emph{flow}, of a smooth 
vector field $X$, which may depend on time, defined on $\RR\times M$ (or on an open subset of
$\RR\times M$) is the map $\Psi^X$, defined on an open subset of $\RR\times\RR\times M$,
taking its values in $M$, such that for each $t_0\in\RR$ and each $x_0\in M$,
the maximal solution $\varphi$ of the differential equation determined by $X$ which
satisfies $\varphi(t_0)=x_0$ is the map $t\mapsto \Psi^X(t,t_0,x_0)$. When $X$ 
does not depend on time, $\Psi^X(t,t_0,x_0)$ only depends on $t-t_0$ and $x_0$.
So instead of the full flow $\Psi^X$, one can use the \emph{reduced flow} $\Phi^X$, defined
on an open subset of $\RR\times M$ by the equality  $\Phi^X(t,x_0)=\Psi^X(t,0,x_0)$.
One often write $\Phi^X_t(x_0)$ to emphasize the fact that $\Phi^X_t$ is a diffeomorphism
between two open subsets of $M$.} 
of $X_H$. If, at a time $t_0$, the state of the dynamical system described by $X_H$
is a perfectly defined point $x_0\in M$, the state of the system, at any other 
time $t_1$ at which it exists, is the point $x_1=\Phi^{X_H}_{t_1-t_0}(x_0)$.
\par\smallskip

Let us assume that $\mu(t_0)$ is the statistical state of 
such a system at a given time $t_0$. We assume, for simplicity, that $\mu(t_0)$ is 
smooth and we denote by $\rho(t_0)$ its probability density with respect to the 
Liouville measure $\lambda_\omega$. Let $t_1$ be another time at which the 
considered system still exists. The reduced flow $\Phi^{X_H}$ of the Hamiltonian 
vector field $X_H$ is such that $\Phi^{X_H}_{t_1-t_0}$ is a symplectic diffeomorphism
of an open subset of $M$ onto another open subset of this manifold, whose inverse is
$\Phi^{X_H}_{t_0-t_1}$.
The statistical state of the system at time $t_1$ is therefore smooth, with a 
probability density $\rho(t_1)$ with respect to $\lambda_\omega$, related to 
$\rho(t_0)$ by the equation
 $$\rho(t_1)=\rho\bigl(t_0)\circ\Phi^{X_H}_{t_0-t_1}\,.
 $$
In other words, for any $x\in M$, 
 $$\rho(t_1, x)=\rho\bigl(t_0,\Phi^{X_H}_{t_0-t_1}(x)\bigr)\,.
 $$

\end{enonce}

\begin{defi}
\label{EntropyDefinition} 
Let $\rho$ be the probability density, with respect to the Liouville measure
$\lambda_\omega$, of a continuous statistical state on the symplectic manifold $(M,\omega)$.
The \emph{entropy} of this statistical state, denoted by $s(\rho)$, is defined as
follows. With the convention that when $x\in M$ is such that $\rho(x)=0$, we set
$\displaystyle \log\left(\frac{1}{\rho(x)}\right)\rho(x)=0$, we can consider
$\displaystyle x\mapsto\log\left(\frac{1}{\rho(x)}\right)\rho(x)$ as a continuous function 
well defined on $M$, taking its values in $\RR$. When this function
is integrable on $M$ with respect to the Liouville measure $\lambda_\omega$, we set 
 $$s(\rho)=\int_M\log\left(\frac{1}{\rho(x)}\right)\rho(x)\lambda_\omega(\d x)
          =-\int_M\log\bigl(\rho(x)\bigr)\rho(x)\lambda_\omega(\d x)\,.
 $$
Otherwise, we set
 $$s(\rho)=-\infty\,.
 $$  
The map $\rho\mapsto s(\rho)$ so defined on the set of all continuous probability densities
on $M$ is called the \emph{entropy functional}.           
\end{defi}
\goodbreak

\begin{enonce}[definition]{Comments about entropy \phantom{bla bla bla bla bla bla}}
\label{CommentsEntropy}
\noindent
{\rm 1.}\quad The concept of \emph{entropy} is due to Rudolf Clausius, who used it to
formulate precisely the second principle of thermodynamics.

\par\smallskip\noindent
{\rm 2}\quad The entropy of a real system in Physics is always positive. The third law 
of thermodynamics states that the entropy of a system in thermodynamic equilibrium, 
when its state of minimal energy is unique, decreases towards $0$ when its absolute
temperature decreases towards $0$ degree Kelvin. Physicists therefore consider as 
an unacceptable anomaly the fact that the entropy functional can take negative values,
and are scandalized at the sight of $-\infty$ as a possible value of entropy. Indeed,
such a value is in clear conflict with Heisenberg's principle of uncertainty.
The von Neumann entropy\footnote{In quantum statistical mechanics, the von 
Neumann entropy of a state mathematically described by a \emph{density matrix} $\rho$
is the trace of $-\rho\ln\rho$. It was defined and extensively used by the
Hungarian-American universal scientist John von Neumann (1903--1957).}, 
used in quantum statistical mechanics, is always positive, and the
entropy defined in \ref{EntropyDefinition} is only its imperfect classical approximation.

\par\smallskip\noindent
{\rm 3.}\quad   
In his famous paper \cite{Shannon1948}, written during the second world war
and published in 1948,  the American mathematician, electrical engineer and
cryptographer Claude Elwood Shannon (1916--2001) laid the foundations
of information theory. He defined in this paper a concept of \emph{entropy} 
whose \emph{opposite} can be used as measurement of the information contained 
in a message, and considered its evolution when the message is transmitted through 
a telecommunications channel. Curiously enough, by reference to Boltzmann's works,
the notation he used for his entropy is the letter $H$, although he observed that 
his entropy's expression is similar to the expression of the \emph{opposite} 
of Boltzmann's H-function. For a random variable $X$ which can take $N$ 
possible values $x_i$, repectively with the probabilities 
$k_i$\footnote{The notation used by Shannon for the probability of $x_i$ is
$p_i$, $1\leq i\leq N$. Here I use $k_i$ instead to avoid any risk of confusion with 
the Darboux coordinates $p_i$ in a canonical chart of a symplectic manifold.} 
$(1\leq i\leq n)$, the $k_i$ satisfying $k_i\geq 0$ and $\sum_{i=1}^Nk_i=1$, 
Shannon defined its entropy $H(X)$ by stating
 $$H(X)=\sum_{i=1}^N \log\left(\frac{1}{k_i}\right)k_i=-\sum_{i=1}^N(\log k_i)k_i\,,
 $$ 
with the usual convention $0\log 0=0$. In Appendix 2 of his above cited paper, 
page 49, he proved that up to multiplication by a strictly positive constant, 
his entropy is the only function  which satisfies the following three very 
reasonable requirements.
\par\smallskip

\begin{itemize}

\item{} The function $H$ must continuously depend on the probabilities $k_i$, 
$1\leq i\leq N$.

\item{} When the $k_i$ are all equal to $1/N$ the function
$N\mapsto H(1/N,\ldots,1/N)$ ($N$ terms) must increase monotonically with $N$.

\item{} When some possible values of the random variable $X$ are obtained as the 
result of two successive choices, the value of $H(X)$ must be equal
to the weighted sum of the individual values of $H$. For example, for a random
variable $X$ with the three possible values~:  $x_1$ with probability $k_1=1/2$, $x_2$ with
probability $k_2=1/3$ and $x_3$ with probability $k_3=1/6$, the values $x_1$, $x_2$
and $x_3$ can be obtained in two steps. In the first step, a first trial is done
in which one looks at the value taken by a random variable $Y$ with two possible values, 
$y_1$ and $y_2$, both obtained with probability $1/2$. In the second step, 
if the value taken by $Y$ is $y_1$, one states that the value taken by $X$ is $x_1$; 
if the value taken by $Y$ is $y_2$, one looks at the value taken by a random variable
$Z$ with two possible values, $z_1$ with probability $2/3$ and $z_2$ with probability
$1/3$. If the value taken by $Z$ is $z_1$, one states that the value taken by $X$
is $x_2$, and if the value taken by $Z$ is $z_2$, one states that the value taken by $X$ 
is $x_3$. The equality that the function $H$ is required to satisfy is
 $$H\left(\frac{1}{2},\frac{1}{3},\frac{1}{6}\right)=H\left(\frac{1}{2},\frac{1}{2}\right)
    +\frac{1}{2}H\left(\frac{2}{3},\frac{1}{3}\right)\,.
 $$ 
\end{itemize}
    
\par\smallskip 
Interested readers are referred to Alain Chenciner's paper
\cite{Chenciner2017} for a more detailed account of Claude Shannon's works and their influence 
on today's science.

\par\smallskip\noindent
{\rm 4.}\quad The American physicist Edwin Thompson Jaynes (1922--1998) observed,
in \cite{Jaynes1963, Jaynes1968} (see also \cite{LimitingDensityWikipedia}), 
that the definition \ref{EntropyDefinition} of entropy
for a continuous statistical state of probability density $\rho$ 
with respect to the Liouville measure, 
 $$s(\rho)=\int_M\log\left(\frac{1}{\rho(x)}\right)\rho(x)\lambda_\omega(\d x)
          =-\int_M\log\bigl(\rho(x)\bigr)\rho(x)\lambda_\omega(\d x)\,,
 $$
\emph{is not} a correct adaptation of Shannon's entropy
for a discrete statistical state which can take $N$ distinct values $x_i$, 
with the respective probabilities $k_i$, 
 $$H(X)=-\sum_{i=1}^N(\log k_i)k_i\,,\quad \hbox{with}\ 
          k_i\geq 0\ \hbox{for all}\ i\in\{1,\ldots,N\}\ 
           \hbox{and}\ \sum_{i=1}^Nk_i=1\,. 
 $$  
While $H(X)$ is always a dimensionless number satisfying $H(X)\geq 0$, 
and $H(X) =0$ if and only if there exists only one integer $i\in\{1,\ldots,N\}$ 
such that $k_i=1$, all other $k_j$, fo $j\neq i$, being equal to $0$, the above 
expression  of $s(\rho)$ depends on the chosen units. Indeed in this expression, 
while $\rho(x)\lambda_\omega(\d x)$ is dimensionless, $\rho (x)$, as well as 
$\lambda_\omega(\d x)$ are not dimensionless. A change of the units (of length, 
time and mass) changes the value of the the term  $\log\bigl(\rho(x)\bigr)$ by 
addition of a constant, which can be either positive or negative.  
Therefore when one uses definition 
\ref{EntropyDefinition}, the sign of entropy does not have any physical meaning. 
In the above cited papers of Jaynes, the author considered problems in 
statistics more general than those encountered in classical statistical mechanics, 
in which the Liouville measure may not be available. He proposed to replace, in 
the expression of the entropy $s(\rho)$, the term 
$\displaystyle\log\left(\frac{1}{\rho(x)}\right)$
by $\displaystyle\log\left(\frac{m(x)}{\rho(x)}\right)$, where $m(x)$ is the 
probability density of a reference statistical state with respect to which the entropy
$s(\rho)$ is evaluated. Of course the probability densities $m(x)$ and $\rho(x)$
must be taken with respect to the same measure. In the framework of classical 
statistical mechanics, this measure is the Liouville measure $\lambda_\omega$,
so the correction proposed by Jaynes can be written
 $$s_{\rm Jaynes}(\rho)=\int_M\log\left(\frac{m(x)}{\rho(x)}\right)
    \rho(x)\lambda_\omega(\d x)\,.
 $$
Probably because he considered problems in which the Liouville measure did not 
appear, Jaynes did not clearly state how $m(x)$ should be chosen, although he 
recommanded the use of a probability density invariant by the group of automorphisms 
of the considered measurable space. Therefore it seems that in the framework of 
classical statistical mechanics, when  the support $W$ of 
$\rho$\footnote{The support of $\rho$ is the closure of the subset of $M$ made 
of points $x\in M$ such that $\rho(x)\neq 0$.}  is of finite 
$\lambda_\omega$-measure, one should use the following probability density~:
 $$m(x)=\begin{cases}
         \displaystyle\frac{1}{\lambda_\omega(W)}&\text{when $x\in W$},\\
          0&\text{when $x\notin W$}.
          \end{cases}
 $$
With this choice of $m$, $s(\rho)$ and $s_{\rm Jaynes}(\rho)$
are related by
 $$s_{\rm Jaynes}(\rho)=s(\rho)-\log\bigl(\lambda_\omega(W)\bigr)\,.
 $$
The corrected entropy $s_{\rm Jaynes}(\rho)$ proposed by Jaynes is dimensionless.
It differs from the entropy $s(\rho)$ of definition \ref{EntropyDefinition} only
by a constant, which depends on the units chosen for time, length and mass, 
and can take negative as well as positive values.
\par\smallskip

In calculus of variations, it may be useful to consider infinitesimal 
variations of $\rho$ whose support does not always remain contained in the 
support of $\rho$. Instead of the support of $\rho$, one should take for $W$, 
in the above formula, an open subset of $M$ of finite $\lambda_\omega$-measure 
which contains the support of $\rho$.
 
\par\smallskip\noindent
{\rm 5.}\quad 
For a better understanding of how the entropies of continuous and discrete 
statistical states are related, let us consider the process of discretization 
of a continuous statistical state. As above, we assume that the support $W$ of 
the probability density $\rho$ is of finite $\lambda_\omega$-measure. For 
simplicity\footnote{These assumptions could probably be avoided with the use 
of more sophisticated concepts in integration theory, such as the Stieltjes 
integral, so named in honour of the Dutch mathematician Thomas Joannes Stieltjes 
(1856--1892).}, we moreover assume that $\rho_{\rm max}=\sup_{x\in M}\rho(x)$ 
too is finite and that, for each real $r$
satisfying $0\leq r\leq\rho_{\rm max}$, 
 $$\lambda_\omega\Bigl(\bigl\{x\in W
    \bigm|\rho(x)=r\bigr\}\Bigr)=0\,.
 $$ 
For each $r\geq 0$, let us set
 $$G(r)=\lambda_\omega\Bigl(\bigl\{x\in W\bigm|0\leq\rho(x)\leq r\bigr\}\Bigr)\,.
 $$
Then $G$ is a continuous and monotonically increasing function which takes all 
values in the closed interval $[0,\lambda_\omega(W)]$.  Let $N$ be an integer 
satisfying $N>2$. There exist $N$ real numbers $r^N_i$, $1\leq i\leq N$, such that
for each $i\in\{1,\ldots,N\}$
 $$G(r^N_i)=\frac{i\lambda_\omega(W)}{N}\,.
 $$
We set
 $$V^N_1=\bigl\{x\in W\bigm|0\leq\rho(x)\leq r_1\bigr\}\,,
 $$
and, for each $i\in\{2,\ldots,N\}$,
 $$V^N_i=\bigl\{x\in W\bigm|r_{i-1}<\rho(x)\leq r_i\bigr\}\,.
 $$
The $V^N_i$ are measurable subsets of $M$ which satisfy, 
for $1\leq i,j\leq N$,
 $$\lambda_\omega(V^N_i)=\frac{\lambda_\omega(W)}{N}\,,\quad 
V^N_i\cap V^N_j=\emptyset\ \hbox{if}\ i\neq j\,,
                \quad \bigcup_{i=1}^NV^N_i=W\,.
 $$
Now we set, for each $i\in\{1,\ldots,N\}$,
 $$k^N_i=\int_{V^N_i}\rho(x)\lambda_\omega(\d x)\,,\quad
   \rho^N_i=\frac{N k^N_i}{\lambda_\omega(W)}\,.
 $$
We have
 $$0\leq k^N_i\leq 1\ \hbox{for each $i\in\{1,\ldots,N\}$}\,,\quad\sum_{i=1}^Nk^N_i=1\,.
 $$
Let $\rho^N$ be the function defined on $M$ by
 $$\rho^N(x)=\begin{cases}
   \displaystyle\frac{Nk^N_i}{\lambda_\omega(W)}&
    \text{if $x\in V^N_i$, $1\leq i\leq N$,}\\
     0&\text{if $x\notin\bigcup_{i=1}^NV^N_i=W$}\,.
      \end{cases}
 $$
The function $\rho^N$ is everywhere $\geq 0$ on $M$, and only takes $N$
distinct non-zero values. It is a discrete approximation of the probability density
$\rho$, which satisfies
 $$\int_M\rho^N(x)\lambda_\omega(\d x)=\sum_{i=1}^Nk^N_i=1\,.
 $$
The function $\rho^N$ is therefore the probability density of a statistical state
on $M$. Although it is not continuous, we can use \ref{EntropyDefinition} 
to calculate $s(\rho^N)$. We obtain
 $$s(\rho^N)=\sum_{i=1}^Nk^N_i(-\log k^N_i)+\log\bigl(\lambda_\omega(W)\bigr)-\log N\,.
 $$
We observe that the term $\sum_{i=1}^Nk_i(-\log k_i)$ is the Shannon entropy $H(X^N)$
of a random variable $X^N$ wich can take $N$ distinct values, for example the values
$1,\ldots,N$, with the respective probabilities $k^N_1,\ldots,k^N_N$. So we can write
 $$H(X^N)=s(\rho^N)-\log\bigl(\lambda_\omega(W)\bigr)+\log N
         =s_{\rm Jaynes}(\rho^N)+\log N\,.
 $$
When $N\to+\infty$, $s_{\rm Jaynes}(\rho^N)\to s_{\rm Jaynes}(\rho)$ and 
$\log N\to +\infty$. The above equality proves that when $N\to+\infty$, the
Shannon entropy of the discrete approximation, by a random variable $X^N$ which 
can take $N$ distinct non-zero values, of the continuous statistical state
of probability density $\rho$, does not remain bouded  and increases as fast 
as $\log N$.
   
\par\smallskip\noindent
{\rm 6.}\quad
During the years 1950--1960, several scientists, notably
Edwin Thompson Jaynes cited above (see also \cite{Jaynes1957a, Jaynes1957b}
by the same author) 
and the American mathematician George Whitelaw Mackey (1916--2006) \cite{Mackey1963}, 
proposed the use of information theory in thermodynamics. 
\par\smallskip

Interested readers are referred to Roger Balian's paper \cite{Balian2005}, 
in which they will find a clear account
of the use of probability concepts in physics and of information 
theory in quantum mechanics.  
\end{enonce}

\begin{prop}
\label{ConstancyOfEntropy}
On a symplectic manifold $(M,\omega)$, we consider a smooth  Hamiltonian 
$H\in C^\infty(M,\RR)$ which does not depend on time. Let $X_H$ be the 
associated Hamiltonian vector field on $M$. Let $\rho(t_0)$ be the probability density
of a smooth statistical state of the dynamical system determined by $X_H$ at a time
$t_0$. The probability density $\rho(t_1)$ of the statistical state of the system 
at any other time $t_1$ at which the system still exists is such that
 $$s\bigl(\rho(t_1)\bigr)=s\bigl(\rho(t_0)\bigr)\,.
 $$  
In other words, the entropy of the statistical state of the system remains constant 
as long as this statistical state exists. 
\end{prop}

\begin{proof}
As seen in \ref{EvolutionStatisticalState},
$\rho(t_1)=\rho(t_0)\circ\Phi^{X_H}_{t_0-t_1}$, so for each $x\in M$,
 $$\rho(t_0,x)=\rho\bigl(t_1,\Phi^{X_H}_{t_1-t_0}(x)\bigr)\,.
 $$ 
When $s\bigl(\rho(t_0)\bigr)\neq-\infty$, we can write
 \begin{align*}
  s\bigl(\rho(t_0)\bigr)
           &=\int_M\log\left(\frac{1}{\rho(t_0,x)}\right)\rho(t_0,x)\lambda_\omega(\d x)\\
           &=\int_M\log\left(\frac{1}{\rho\bigl(t_1,\Phi^{X_H}_{t_1-t_0}(x)\bigr)}\right)
                    \rho\bigl(t_1,\Phi^{X_H}_{t_1-t_0}(x)\bigr)(\lambda_\omega)(\d x)\\
           &=\int_M\log\left(\frac{1}{\rho(t_1,y)}\right)
                    \rho(t_1,y)(\Phi^{X_H}_{t_1-t_0})_*(\lambda_\omega)(\d y)  \\
           &=\int_M\log\left(\frac{1}{\rho(t_1,y)}\right)\rho(t_1,y)\lambda_\omega(\d y)\\
           &=s\bigl(\rho(t_1)\bigr)\,, 
 \end{align*}
where we have used the change of integration variable $y=\Phi^{X_H}_{t_1-t_0}(x)$ 
and the invariance of the Liouville measure by symplectomorphism (\ref{LiouvilleMeasure}).
When $s\bigl(\rho(t_0)\bigr)=-\infty$, the same calculation leads to a divergent integral
for the expression of $s\bigl(\rho(t_1)\bigr)$, which therefore is equal to $-\infty$.
\end{proof}

\subsection{Gibbs states for a Hamiltonian system}
\label{subsec:2.3} 
In this subsection, $H$ is a smooth Hamiltonian which does not depend on time,
defined on a symplectic manifold $(M,\omega)$, and $X_H$ is the associated Hamiltonian
vector field. The \emph{Gibbs states} defined here are built with the Hamiltonian $H$
as the only conserved quantity. Their main properties are briefly indicated. Gibbs state
for the Hamiltonian action of a Lie group are considered in Section \ref{sec:3}

\begin{prop}\label{StationarityOfEntropy}
Under the assumptions and with the notations of \ref{subsec:2.3}, let $\rho$ be the
probability density, with respect to the Liouville measure $\lambda_\omega$, 
of a smooth statistical state on $M$. We assume that $\rho$ is such that the integrals
which define the entropy $s(\rho)$ (definition \ref{EntropyDefinition} ) 
and the mean value ${\mathcal E}_\rho(H)$ of the Hamiltonian $H$ (definition 
\ref{DefiStatisticalState}) are convergent and can be differentiated under the sign $\int$
with respect to infinitesimal variations of $\rho$. 
The entropy function $s$ is stationary at $\rho$
with respect to smooth infinitesimal variations of $\rho$ which leave fixed the 
mean value of $H$ if and only if there exists a real $\beta\in\RR$ such that, for every $x\in M$,
 $$\rho(x)=\frac{1}{P(\beta)}\exp\bigl(-\beta H(x)\bigr)\,,\quad\hbox{with}\quad
      P(\beta)=\int_M\exp\bigl(-\beta H(x)\bigr)\lambda_\omega(\d x)\,.
 $$
\end{prop}

\begin{proof}
Let $\tau\mapsto\rho_\tau$ be a smooth infinitesimal variation of 
$\rho$ which leaves fixed the mean value of $H$. Since 
$\displaystyle \int_M\rho_\tau(x)\lambda_\omega(\d x)$ and 
$\displaystyle\int_M\rho_\tau(x)H(x)\lambda_\omega(\d x)$ do not depend on $\tau$, 
it satisfies, for all $\tau\in]-\varepsilon,\varepsilon[$\,,
 $$\int_M\frac{\partial\rho(\tau,x)}{\partial\tau}\lambda_\omega(\d x)=0\,,
   \int_M\frac{\partial\rho(\tau,x)}{\partial\tau}H(x)\lambda_\omega(\d x)=0\,.
 $$
Moreover an easy calculation leads to 
 $$\frac{\d s(\rho_\tau)}{\d \tau}\Bigm|_{\tau=0}=-\int_M
    \frac{\partial\rho(\tau,x)}{\partial\tau}\Bigm|_{\tau=0}
     (1+\log\bigl(\rho(x)\bigr)\lambda_\omega(\d x)\,.
 $$
By a well known result in calculus of variations, this implies 
that the entropy functional is stationary at $\rho$
with respect to smooth infinitesimal variations of $\rho$ which leave fixed 
the mean value of $H$, if and only if there exist two real constants $\alpha$ 
and  $\beta$, the \emph{Lagrange multipliers}, such that, for every $x\in M$,
 $$1+\log(\rho(x))+\alpha+\beta H(x)=0\,,
 $$
which leads to
 $$\rho(x)=\exp\bigl(-1-\alpha-\beta H(x)\bigr)\,.
 $$
By writing that $\displaystyle\int_M\rho(x)\lambda_\omega(\d x)=1$, 
we see that $\alpha$ is determined by $\beta$:
 $$\exp(1+\alpha)=P(\beta)=\int_M\exp\bigl(-\beta H(x)\bigr)\lambda_\omega(\d x)\,.\eqno{\qedhere}
 $$
\end{proof}

\begin{defis}\label{DefiGibbsStatePartitionFunction} Let $\beta\in\RR$ be a real 
which satisfies the conditions of proposition \ref{StationarityOfEntropy}. The
smooth statistical state whose probability density, with respect to the Liouville 
measure $\lambda_\omega$, is
 $$\rho_\beta(x)=\frac{1}{P(\beta)}\exp\bigl(-\beta H(x)\bigr)\,,\quad x\in M\,,
 $$
with
 $$P(\beta)=\int_M\exp\bigl(-\beta H(x)\bigr)\lambda_\omega(\d x)\,,
 $$
is called the \emph{Gibbs state} associated to (or indexed by) $\beta$. The function
$P$ of the real variable $\beta$ is called the \emph{partition function} of the 
dynamical system determined by the Hamiltonian vector field $X_H$.
\end{defis}

\begin{prop}\label{InvarianceOfGibbsState}
Let $\beta\in\RR$ be a real which satisfies the conditions of proposition 
\ref{StationarityOfEntropy}. The probability density $\rho_\beta$ of the 
corresponding Gibbs state (definition \ref{DefiGibbsStatePartitionFunction}) 
remains invariant under the flow of the Hamiltonian vector field $X_H$.
\end{prop}

\begin{proof}
Since the Hamiltonian $H$ does not depend on time, it is a first integral of
the differential equation determined by $X_H$, \emph{i.e.}, it keeps a constant 
value on each integral curve of $X_H$. Therefore $\rho_\beta$ keeps a constant 
value on each integral curve of $X_H$.
\end{proof}

\begin{enonce}[definition]{Some properties of Gibbs states}
\label{PropertiesOfGibbsStates}
We have seen (\ref{StationarityOfEntropy}) that the entropy functional
$s$ is stationary at each Gibbs state with respect to all infinitesimal variations 
of its probability density which leave invariant the mean value of the Hamiltonian $H$. 
A stronger result holds: given any Gibbs state of probability density $\rho_\beta$, 
on the set of all continuous statistical states whose probability density $\rho$
is such that ${\mathcal E}_\rho(H)={\mathcal E}_{\rho_\beta}(H)$, the entropy 
functional $s$ reaches its only strict maximum at the Gibbs state of probability 
density $\rho_\beta$.
\par\smallskip

When the set $\Omega$ of reals $\beta$ for which a Gibbs state indexed by 
$\beta$ exists is not empty, this set is an open interval \,$]\,a,b\,[$\, of $\RR$,
where either $a\in \RR$, or $a=-\infty$, either $b\in\RR$ and $b>a$, or $b=+\infty$.
When in addition $H$ is bounded from below, \emph{i.e.}, when there exists
$m\in\RR$ such that, for any $x\in M$, $m\leq H(x)$, the open interval $\Omega$
is unbound on the right side, \emph{i.e.}, $\Omega=\,]\,a,+\infty\,[$  where 
either $a\in\RR$, or $a=-\infty$.
\par\smallskip

 We have already defined on $\Omega$ the 
partition function $P$ (\ref{DefiGibbsStatePartitionFunction}). Other functions 
can be defined on $\Omega$ as follows. For each $\beta\in\Omega$, the entropy
$s(\rho_\beta)$ exists of course, and one can prove that the mean value 
${\mathcal E}_{\rho_\beta}(H)$ (definition \ref{DefiStatisticalState}) 
of the Hamiltonian $H$, in the Gibbs state indexed by $\beta$, exists too, 
as well as ${\mathcal E}_{\rho_\beta}(H^2)$ and 
${\mathcal E}_{\rho_\beta}\Bigl(\bigl(H-{\mathcal E}_{\rho_\beta}(H)\bigr)^2\Bigr)$. 
So we can set
 $$S(\beta)=s(\rho_\beta)\,,\quad 
   E(\beta)={\mathcal E}_{\rho_\beta}(H)\,,\quad \beta\in \Omega\,.
 $$
The functions $P$ (partition function), $E$ (mean value of the Hamiltonian, 
considered by physicists as the energy) and $S$ (entropy) so defined are of class 
$C^\infty$ on $\Omega$ and 
satisfy, for any $\beta\in\Omega$,
 \begin{align*}
  P(\beta)&>0\,,\\ 
  E(\beta)&=-\frac{1}{P(\beta)}\frac{\d P(\beta)}{\d\beta}
          =\frac{\d\bigl(-\log P(\beta)\bigr)}{\d\beta}\,,\\
 \frac{\d E(\beta)}{\d \beta}&=\frac{\d^2\bigl(-\log P(\beta)\bigr)}{\d\beta^2}
          =-{\mathcal E}_{\rho_\beta}
  \Bigl(\bigl(H-{\mathcal E}_{\rho_\beta}(H)\bigr)^2\Bigr)\,,
 \end{align*}
 \begin{align*}
 S(\beta)&=\log P(\beta)+\beta E(\beta)
           =\beta\frac{\d\bigl(-\log P(\beta)\bigr)}{\d\beta}-\bigl(-\log P(\beta)\bigr)\,,\\
 \frac{\d S(\beta)}{\d \beta}&=\beta\frac{\d E(\beta)}{\d \beta}\,.
\end{align*}   
The above expression of $\displaystyle\frac{\d E(\beta)}{\d \beta}$ proves that
$\beta\mapsto E(\beta)$ is a non-increasing function. When the Hamiltonian $H$ is 
not a constant, for each $\beta\in\Omega$, the continuous function defined on $M$
$x\mapsto \bigl(H(x)-{\mathcal E}_{\rho_\beta}(H)\bigr)^2$ takes its values in 
$\RR^+$ and is not always equal to $0$. Its mean value 
${\mathcal E}_{\rho_\beta}\Bigl(\bigl(H-{\mathcal E}_{\rho_\beta}(H)\bigr)^2\Bigr)$ 
is therefore $>0$, which proves that $\beta\mapsto E(\beta)$ is a strictly 
decreasing function on $\Omega$. The map $E$ is open, and is a diffeomorphism of
$\Omega$ onto its image $\Omega^*$.
\par\smallskip

The above expression of $S(\beta)$ shows that the functions 
$\beta\mapsto -\log\bigl(P(\beta)\bigr)$ and $\beta\mapsto S(b)$ are 
\emph{Legendre transforms} of each other. They are indeed linked by the same 
relation as that which, in calculus of variations, links a hyper-regular Lagrangian
with the associated energy. Here the hyper-regular \lq\lq Lagrangian\rq\rq, defined on $\Omega$, is 
$\beta\mapsto -\log P(\beta)$, the Legendre map is the diffeomorphism $E:\Omega\to\Omega^*$, 
the \lq\lq energy\rq\rq, defined on $\Omega$, is $\beta\mapsto S(\beta)$, and the
\lq\lq Hamiltonian\rq\rq, defined on $\Omega^*$, is $S\circ E^{-1}$. By using the 
above expression of $\displaystyle\frac{\d S(\beta)}{\d \beta}$, we can write
 $$E^{-1}(e)=\frac{\d\bigl(S\circ E^{-1}(e)\bigr)}{\d e}\,,\quad e\in\Omega^*\,. 
 $$
As soon as 1869, the Legendre transform was used in thermodynamics by the French 
scientist François Massieu \cite{Balian2015, Massieu1869a, Massieu1869b, Massieu1876}.   
\par\smallskip

The results stated here without proof are proven below, in a more general setting,
for    Gibbs states (\ref{DiffPartitionFunction}, \ref{MaximalityEntropy}, 
\ref{EJDiffeoRmks}).
\end{enonce}

\begin{enonce}[definition]{Gibbs states, temperatures and thermodynamic equilibria}
\label{ThermodynamicEquilibria}

Let us now assume that the dynamical system determined by the Hamiltonian vector 
field $X_H$ mathematically describes the evolution with time of a physical system, 
an object of the real world. Physicists consider each Gibbs state of the considered 
dynamical system, indexed by some $\beta\in\RR$, as the mathematical description of a 
\emph{state of thermodynamic equilibrium} of the corresponding physical system, 
and $\beta$ as a quantity related to the \emph{absolute temperature} 
$T$ of the physical system by the equality
 $$\beta=\frac{1}{kT}\,,\eqno(*)
 $$  
where $k$ is a constant which depends on the chosen units, 
called \emph{Boltzmann's constant}. This identification of Gibbs states with 
thermodynamic equilibria is justified by the following property. 
Let us consider two similar physical systems, mathematically described by two
Hamiltonian systems, whose Hamiltonians are, respectively, $H_1$ defined on the 
symplectic manifold $(M_1,\omega_1)$ and $H_2$ defined on the symplectic manifold 
$(M_2,\omega_2)$. We first assume that they are independent and both in a Gibbs 
state. We denote by $\rho_{1,\beta_1}$ and $\rho_{2,\beta_2}$ the probability 
densities of the Gibbs states, indexed by the reals $\beta_1$ and $\beta_2$, 
in which these two systems are, respectively. Let $E_1(\beta_1)$ and 
$E_2(\beta_2)$ be the corressponding mean values of their Hamiltonians. Let us 
now assume that the two systems are coupled in a way allowing an exchange of 
energy between them. For example, the corresponding objects of the real world 
can be two vessels containing a gas, separated by a wall allowing a transfer 
of heat between them. Coupled together, they make a new physical system, 
mathematically described by a Hamiltonian system on the symplectic manifold 
$(M1\times M_2,\ \omega_{\rm new}=p_1^*\omega_1+p_2^*\omega_2)$, where
$p_1:M_1\times M_2\to M_1$ and $p_2:M_1\times M_2\to M_2$ are the canonical 
projections. The Hamiltonian of this new system can be made as close to
$H_1\circ p_1+H_2\circ p_2$ as one wishes, by making very small the coupling 
between the two systems. We can therefore consider $H_1\circ p_1+H_2\circ p_2$
as a reasonable approximation of the Hamiltonian of the new system.
When the two subsystems are in the Gibbs states indexed, respectively, by $\beta_1$
and by $\beta_2$, the new system made of these two coupled 
subsystems is in the statistical state of probability density
$\rho_{1,\beta_1}\circ p_1+\rho_{2,\beta_2}\circ p_2$, and its entropy is
$S_1(\beta_1)+S_2(\beta_2)$. If $\beta_1\neq\beta_2$, the new system is not
in a Gibbs state. Let us indeed assume, for example, that $\beta_1<\beta_2$. 
If a transfer of energy between the two subsystems occurs, in which the 
energy of the first subsystem decreases while the energy of the second 
subsystem increases by an equal amount, the modified Gibbs state of the first
subsystem becomes indexed by $\beta'_1>\beta_1$ and that of the second subsystem 
by $\beta'_2<\beta_2$ since, as seen in  \ref{PropertiesOfGibbsStates}, for 
$i=1$ as well as for $i=2$, we have
$\displaystyle\frac{\d E_i(\beta'_i)}{\d\beta'_i}<0$. As long as 
$\beta'_1<\beta'_2$, such an energy transfer between the two subsystems results 
in an increase of the entropy of the total new system, until 
$\beta'_1=\beta'_2=\beta_n$, which indexes the Gibbs state of the new system 
for a mean value of its Hamiltonian $E_1(\beta_1)+E_2(\beta_2)$.
We have of course $\beta_1<\beta_n<\beta_2$, which proves that when the state of 
the new system evolves from its initial state towards its Gibbs state, the 
energy flow goes from the subsystem whose Gibbs initial state is indexed by the smaller
$\beta_1$ towards the subsystem whose initial Gibbs sate in indexed by the larger
$\beta_2$. This result is in agreement with everydays's experience, since 
equality $(*)$ implies that when $\beta>0$, a smaller value of $\beta$ 
corresponds to a higher temperature.
\end{enonce}

\begin{enonce}[definition]{Evolution towards a thermodynamic equilibrium}
\label{TowardsThermodynamicEquilibrium}
In the real world, the state of an approximately isolated system often evolves 
with time towards a state of thermodynamic equilibrium. When such a state is 
approximately reached, it remains approximately stationary, with
small fluctuations. Let us mathematically modelize this evolution as 
the variation with time of the statistical state of a Hamiltonian dynamical system, 
whose smooth Hamiltonian $H$ is defined on a very high-dimensional symplectic 
manifold $(M,\omega)$. Proposition \ref{InvarianceOfGibbsState} above, 
which states that Gibbs states do not change with time, seems to be in reasonably 
good agreement with the identification of Gibbs states with thermodynamic 
equilibria, although it does not explain fluctuations which are experimentally 
obseved in thermodynamic equilibria. On the contrary, proposition 
\ref{ConstancyOfEntropy} above, which states that as long as any smooth 
statistical state exists, its entropy remains constant, is in clear 
disagreement with the behaviour of isolated systems in the real world.
The mathematical description of the evolution with time of a real physical system
by the dynamical system determined by the Hamiltonian vector field of a smooth 
Hamiltonian which does not depend on time, together with definition 
\ref{EntropyDefinition} of the entropy, can be used only for reversible systems. 
It cannot be used to describe the evolution with time of some statistical states towards
the corresponding Gibbs state.    
\end{enonce}

\section{Gibbs states for Hamiltonian actions of Lie groups}
\label{sec:3}

The general definition of a Gibbs states is very natural~: it amounts to introduce, 
in the definition of a statistical state, not only the Hamiltonian, but other
conserved quantities too, on the same footing as the Hamiltonian. One may even 
forget the Hamiltonian and consider only the moment map of the Hamiltonian action.
\par\smallskip

This idea is already present in the book published by Gibbs in 1902 (\cite{Gibbs1902}, 
chapter I, page 42 and the following pages), the conserved quantities other than 
the Hamiltonian being the components of the total angular momentum. 
\par\smallskip

Following an idea first proposed around 1809 by Joseph Louis Lagrange (1736--1813),
Jean-Marie Souriau defined statistical states on the \emph{manifold of motions} 
of a Hamiltonian dynamical system, instead of on its \emph{phase space}. 
This approach allows a more natural treatment on the same footing of both the 
Hamiltonian and other conserved quantities, because the action of the group of 
translations in time, which may act only locally on the phase space, always acts 
globally on the space of motions. The concept of manifold of motions and its 
properties are presented in subsection \ref{subsec:3.1} below. Gibbs 
states for a Hamiltonian action of a Lie group on a symplectic manifold
are then defined (\ref{GibbsState}), together with generalized temperatures and 
partition functions, and their main property (maximality of entropy) is proven 
(\ref{MaximalityEntropy}). Thermodynamic functions associated to a  Gibbs state 
(mean value of the moment map and entropy) are maps defined on the set of 
generalized temperatures (subsection \ref{subsec:3.2}). The expressions of their 
differentials lead to the definition, on the set of generalized temperatures, 
of a remarkable Riemannian metric, linked to the Fisher-Rao metric of statisticians.
The adjoint action on the set of generalized temperatures is considered in 
subsection \ref{subsec:3.3}, in which the Riemannian metric induced on each 
adjoint orbit is expressed in terms of a symplectic cocycle.
 
\subsection{Symmetries and statistical states}
\label{subsec:3.1}
In this subsection we consider the dynamical system determined on a symplectic
manifold $(M,\omega)$, as explained in \ref{HamiltonianVectorFieldsInMechanics},
by a Hamiltonian vector field $X_H$ whose smooth Hamiltonian $H$, defined on 
$\RR\times M$ or on one of its open subsets, may depend on time.
  
\begin{enonce}[definition]{The manifold of motions of a Hamiltonian system}
\label{ManifoldOfMotions}
Jean-Marie Souriau called \emph{motion} of the dynamical system determined by
a Hamiltonian vector field $X_H$ any maximal solution $\varphi:t\mapsto \varphi(t)$
of Hamilton's differential equation
 $$\frac{\d \varphi(t)}{\d t}=X_H\bigl(t,\varphi(t)\bigr)\,. 
 $$
The \emph{manifold of motions} of the system, denoted by $\Mot(X_H)$, 
is simply the set of all motions, \emph{i.e.}, the set of all
maximal solutions $\varphi$ of the above differential equation. 
It always has the structure of a smooth symplectic manifold. For each 
$t_0\in \RR$, the map $h_{t_0}$ which associates, to each motion 
$\varphi:t\mapsto \varphi(t)$ whose interval of definition contains $t_0$, the point 
$h_{t_0}(\varphi)=\varphi(t_0)\in M$, is indeed, when the subset 
made of motions defined on an interval of $\RR$ which contains $t_0$ is not empty,
a bijection of this subset of $\Mot(X_H)$ onto an open subset of $M$. 
This simple fact allows the definition of a topology and a structure of smooth 
manifold on $\Mot(X_H)$ such that, for each $t_0\in \RR$, 
$h_{t_0}:\varphi\mapsto h_{t_0}(\varphi)=\varphi(t_0)$ 
is a diffeomorphism of the open subset of $\Mot(X_H)$ made of motions defined on 
an interval of $\RR$ which contains $t_0$, onto an open subset of $M$. 
Since the reduced flow of $X_H$ is made of symplectomorphisms, the pull-back 
$h_{t_0}^*\omega$ of the symplectic form $\omega$
does not depend on $t_0$, therefore determines globally a symplectic form 
$\omega_{\Mot(X_H)}$ on the manifold of motions $\Mot(X_H)$. 
\par\smallskip

The manifold $\Mot(X_H)$ may be a \emph{non-Hausdorff}\footnote{We recall that 
a topological space is said to be \emph{Hausdorff} when for each pair
of distinct elements $x$ and $y$ of this space, there exist
neighbourhoods $U$ of $x$ and $V$ of $y$ such that $U\cap V=\emptyset$.
This property is so named in honour of the German mathematician Felix Hausdorff
(1848--1942), an important founder of topology and set theory,
who after losing his Professor position at the university of Bonn, 
was driven to suicide by  the Nazi regime.} manifold, 
although any of its elements has an open Hausdorff neighbourhood symplectomorphic 
to an open subset of $M$.
\par\smallskip

We assume now, until the end of this section, that the Hamiltonian $H$ does not 
depend on time. For a given motion $\varphi\in\Mot(X_H)$, the value of 
$H\bigl(\varphi(t_0)\bigr)$ does not depend on the choice of $t_0$ in the interval on which 
$\varphi$ is defined. Therefore there exists a real-valued, smooth function $H_{\Mot}$, 
defined on $\Mot(X_H)$, such that, for
each $\varphi\in\Mot(X_H)$ and any $t_0$ in the interval on which $\varphi$ is 
defined, $H_{\Mot}(\varphi)=H\bigl(\varphi(t_0)\bigr)$. Since, for each $t_0\in\RR$,
the function $H_{\Mot}$, restricted to the open subset made of motions whose interval 
of definition contains $t_0$, is the pull-back $h_{t_0}^*(H)=H\circ h_{t_0}$ of 
the Hamiltonian $H$, the Hamiltonian vector field $X_{H_{\Mot}}$ on $\Mot(X_H)$,
restricted to this open substet of $\Mot(X_H)$, is the inverse image $h_{t_0}^*(X_H)$
of the Hamiltonian vector field $X_H$ defined on $M$.
\par\smallskip

For any $s\in \RR$ and any motion $\varphi:\,]\,a,b\,[\,\to M\in \Mot(X_H)$, 
defined on the interval $\,]\,a,b\,[\,\subset\RR$, let $\Phi_{\Mot}(s,\varphi)$ be
the parametrized curve, defined on the interval $\,]\,a-s,b-s\,[\,\subset\RR$, 
with  values in $M$,
 $$\Phi_{\Mot}(s,\varphi)(t)=\varphi(t+s)\,,\quad t\in\,]\,a-s, b-s\,[\,.
 $$   
One can easily see that $\Phi_{\Mot}(s,\varphi)\in \Mot(X_H)$ and that the map
 $$\Phi_{\Mot}:\RR\times\Mot(X_H)\to\Mot(X_H)
 $$ 
is a smooth action on the left of the additive Lie group $\RR$ on the manifold 
of motions $\Mot(X_H)$. The infinitesimal generator of this action is the vector 
field on $\Mot(X_H)$, temporarily denoted by $Z$, defined by the equality
 $$Z(\varphi)=\frac{\d\Phi_{\Mot}(s,\varphi)}{\d s}\biggm|_{s=0}\,,
               \quad \varphi\in\Mot(X_H)\,.
 $$
For any real $t_0$ which belongs to the open interval on which $\varphi$ is defined,
we have
 $$Th_{t_0}\bigl(Z(\varphi)\bigr)
    =\frac{\d\Bigl(h_{t_0}\bigl(\Phi_{\Mot}(s,\varphi)\bigr)\Bigr)}{\d s}\biggm|_{s=0}
    =\frac{\d\varphi(t_0+s)}{\d s}\biggm|_{s=0}
    =X_H\bigl(\varphi(t_0)\bigr)\,.
 $$
This result proves that the infinitesimal generator $Z$ of the action $\Phi_{\Mot}$ 
is the Hamiltonian vector field $X_{H_{\Mot}}$. Being generated by the flow of a
Hamiltonian vector field, the action $\varphi$ is therefore Hamiltonian. It admits 
the Hamiltonian $H_{\Mot}$ as a moment map (with the usual convention in which 
the Lie algebra of the additive Lie group $\RR$ is identified with $\RR$ with the 
zero bracket and its dual is too identified with $\RR$, the pairing by duality 
being the usual product of reals).
\par\smallskip

The reader will observe that while the flow of the vector field $X_H$ does not 
always determine a Hamiltonian action of $\RR$ on the symplectic manifold 
$(M,\omega)$, but only a local Hamiltonian action, except when all the motions are defined 
for all $t\in\RR$, the flow of $X_{H_{\Mot}}$ always determines a Hamiltonian 
action of $\RR$ on $(\Mot(X_H),\omega_{\Mot(X_H)})$. However, the price paid for 
obtaining better properties of the flow of a Hamiltonian vector field is the 
fact that $\Mot(X_H)$ can be a non-Hausdorff manifold. 
For this reason, some important results, for example the theorem which asserts
the unicity, for a given initial condition, of a maximal solution of a smooth 
differential equation, can no more be used.
\end{enonce}

In all what follows, the notation $(M,\omega)$ will be used te denote as well the
\emph{phase space} as the \emph{space of motions} of the dynamical system determined
by a smooth Hamiltonian which does not depend on time, according to the context 
in which it is used.   
 
\begin{defis}\label{DefinitionsGeneralizedTemperatures}
Let $\mathfrak g$ be a real, finite-dimensional Lie algebra which acts on a 
connected symplectic manifold $(M,\omega)$ by a Hamiltonian action 
$\varphi:{\mathfrak g}\to A^1(M)$\footnote{For each $k\in\NN$, I denote by $A^k(M)$
the space of fields of $k$-vectors and by $\Omega^k(M)$ the space of 
$k$-exterior differential forms on $M$, with the convention that
$A^0(M)=\Omega^0(M)=C^\infty(M,\RR)$. For $k=1$, $A^1(M)$ is therefore the space 
of smooth vector fields on $M$. Endowed with the Lie bracket as a composition law, 
it is an infinite-dimensional Lie algebra.}.  
Let $J:M\to{\mathfrak g}^*$  be a moment map of the action $\varphi$. For each 
${\bm\beta}\in{\mathfrak g}$, we consider the integral
$$\int_M\exp\bigr(-\langle J(x),{\bm\beta}\rangle\bigr)\lambda_\omega(\d x)\,,\eqno(*)
  $$
where $\lambda_\omega$ is the Liouville measure on $M$.

\par\smallskip\noindent
{\rm 1.}\quad The above integral $(*)$ is said to be \emph{normally convergent}
when there exists an open neihbourhood $U$ of $\bm\beta$ in $\mathfrak g$ and 
a function $f:M\to\RR^+$, integrable on $M$ with respect to the Liouville measure $\lambda_\omega$,
such that for any ${\bm\beta}'\in U$, the following inequality
 $$\exp\bigl(-\langle J(x), {\bm\beta}'\rangle\bigr)\leq f(x)
 $$ 
is satisfied for all $x\in M$.

\par\smallskip\noindent
{\rm 2.}\quad When ${\bm\beta}\in{\mathfrak g}$ is such that the integral $(*)$ above is
normally convergent, $\bm \beta$ is said to be a \emph{generalized temperature}. 
The subset of $\mathfrak g$ made of generalized temperatures will be denoted by $\Omega$.

\par\smallskip\noindent
{\rm 3.}\quad When the set $\Omega$ of generalized temperatures is not empty, the \emph{partition function} 
associated to the Hamiltonian action $\varphi$ is the function $P$ defined on $\Omega$ 
by the equality
 $$P({\bm\beta})=\int_M\exp\bigr(-\langle J(x),{\bm\beta}\rangle\bigr)\lambda_\omega(\d x)\,,
                  \quad x\in M\,,\quad{\bm\beta}\in\Omega\subset{\mathfrak g}\,.
 $$   
\end{defis}

\begin{prop}\label{DiffPartitionFunction} The assumptions and notations are those
of \ref{DefinitionsGeneralizedTemperatures}. The set $\Omega$ of    
generalized temperatures does not depend on the choice of the moment map $J$ of the 
Hamiltonian action $\varphi$. When it is not empty, this set is an open convex 
subset of the Lie algebra $\mathfrak g$, 
the partition function $P$ is of class $C^\infty$ and its differentials of all 
orders can be calculated by differentiation under the integration sign $\int$.
\end{prop}

\begin{proof} When $\bm\beta$ is a generalized temperature, definition 
\ref{DefinitionsGeneralizedTemperatures} implies that there exists a 
neighbourhood $U$ of $\bm\beta$ whose all elements are    
generalized temperatures. When it is not empty, the set $\Omega$ of generalized temperatures 
is therefore open. When the moment map $J$ is replaced by another moment map $J'$, 
the difference $J'-J$ is a constant. The replacement of $J$ by $J'$ has no effect  on the eventual 
normal convergence of the above integral $(*)$, therefore $\Omega$ does not depend on the
choice of the moment map $J$. 
\par\smallskip

Let ${\bm\beta}_0$ and ${\bm\beta}_1$ be two distint elements in $\Omega$ (assumed to be
non-empty), $U_0$ and $U_1$ be neighbourhoods, respectively of  
${\bm\beta}_0$ and ${\bm\beta}_1$, $f_0$ and $f_1$ be the positive functions, defined on $M$
and integrable with respect to the Liouville measure, greater or equal, respectively, 
than the functions $x\mapsto \exp\bigl(-\langle J(x), {\bm\beta}'_0\rangle\bigr)$ and
$x\mapsto \exp\bigl(-\langle J(x), {\bm\beta}'_1\rangle\bigr)$ for all ${\bm\beta}'_0\in U_0$ 
and ${\bm\beta}'_1\in U_1$. For any $\lambda\in[0,1]$, 
$U_\lambda=\{(1-\lambda){\bm\beta}'_0+\lambda {\bm\beta}'_1\bigm|{\bm\beta}'_0\in U_0\,,\ {\bm\beta}'_1\in U_1\}$ 
is a neighbourhood of ${\bm\beta}_\lambda=(1-\lambda){\bm\beta}_0+\lambda {\bm\beta}_1$.
The function $f_\lambda=(1-\lambda)f_0+\lambda f_1$ is integrable on 
$M$. For any ${\bm\beta}'_\lambda\in U_\lambda$, it is greater or equal to the function 
$x\mapsto\exp\bigl(-\langle J(x), {\bm\beta}'_\lambda\rangle\bigr)$. Thereforee
${\bm\beta}_\lambda\in\Omega$, which proves the convexity of $\Omega$. 
\par\smallskip

For each $x\in M$ fixed, the $k$-th differential of $\exp\bigl(-\langle J(x), {\bm\beta}\rangle\bigr)$ 
with respect to ${\bm\beta}$ is
 $$D^k\Bigl(\exp\bigl(-\langle J, {\bm\beta}\rangle\bigr)\Bigr)=(-1)^kJ^{\otimes k}(x)
       \exp\bigl(-\langle J(x), {\bm\beta}\rangle\bigr)\,,
 $$
where $J^{\otimes k}(x)=J(x)\otimes\cdots\otimes J(x)\in ({\mathfrak g}^*)^{\otimes k}$.
Let us recall that $({\mathfrak g}^*)^{\otimes k}$ is canonically isomorphic with the space
${\mathcal L}^k({\mathfrak g},\RR)$ of $k$-multilinear forms on $\mathfrak g$. 
Let us choose any norm on $\mathfrak g$. We take on  
${\mathcal L}^k({\mathfrak g},\RR)$ the \emph{sup} norm. For any
$x\in M$, we have
 $$\Vert J^{\otimes k}(x)\Vert= \sup_{X_i\in{\mathfrak g}\,,\ 
    \Vert X_i\Vert\leq 1\,,\ 1\leq i\leq k}\vert
     \langle J(x),X_1\rangle\cdots\langle J(x),X_k\rangle\vert\,.
 $$
Let ${\bm\beta}\in\Omega$ be a generalized temperature. It follows from the 
definition of a generalized temperature that there exist a real $\varepsilon>0$  and
and a non-negative function $f$ defined on $M$, integrable with respect to the 
Liouville measure and greater than the function 
$x\mapsto \exp\bigl(-\langle J(x), {\bm\beta}'\rangle\bigr)$ 
for any ${\bm\beta}'\in{\mathfrak g}$ satisfying 
$\Vert {\bm\beta}'-{\bm\beta}\Vert\leq\varepsilon$. Let ${\bm\beta}''\in{\mathfrak g}$ 
be such that $\displaystyle\Vert {\bm\beta}''-{\bm\beta}\Vert\leq\frac{\varepsilon}{2}$. 
For all $X_i\in{\mathfrak g}$ satisfying  $\Vert X_i\Vert \leq 1$, with $1\leq i\leq k$,
and any $x\in M$, we have
 $$J^{\otimes k}(x)(X_1,\ldots,X_k)
  =\langle J(x),X_1\rangle\cdots\langle J(x),X_k\rangle\,.
 $$ 
Taking into account the inequality, valid for all $i\in\{1,\ldots k\}$, 
 $$\vert\langle J(x),X_i\rangle\vert\leq\frac{2k}{\varepsilon}
    \exp\left(\frac{\varepsilon}{2k}\vert\langle J(x),X_i\rangle\vert\right)\,,
 $$
we can write
 \begin{multline*}
  \left\vert J^{\otimes k}(x)(X_1,\ldots,X_k)\right\vert\exp\bigl(-\langle J(x),\bm{\beta}''\rangle\bigr)\\
  \leq \left(\frac{2k}{\varepsilon}\right)^k
  \exp\left(-\left\langle J(x),{\bm\beta}''+\frac{\varepsilon}{2k}(\eta_1X_1+\cdots+\eta_kX_k)\right\rangle\right)\,,
 \end{multline*} 
where the terms $\eta_i$, $1\leq i\leq k$, all equal either to $1$ or to $-1$, 
are chosen in such a way that $\bigl\langle J(x), \eta_iX_i\bigr\rangle\leq 0$. 
For each $i\in\{1,\ldots,k\}$, $\vert X_i\vert\leq 1$, therefore
 $$\left\Vert {\bm\beta}-\left({\bm\beta}''+ 
    \frac{\varepsilon}{2k}(\eta_1X_1+\cdots+\eta_kX_k\right)\right\Vert
     \leq\Vert{\bm\beta}-{\bm\beta}''\Vert+
      \frac{\varepsilon k}{2k}\leq\frac{\varepsilon}{2}+
       \frac{\varepsilon}{2}=\varepsilon\,.
 $$ 
We see that
 $$\left\vert J^{\otimes k}(x)(X_1,\ldots,X_k)\right\vert\exp\bigl(-\langle J(x),{\bm\beta}''\rangle\bigr)\leq f(x)\,.
 $$
By taking the upper bound of the left hand side when the $X_i$
take all possible values among elements in $\mathfrak g$ whose norm is smaller than or equal to $1$,
 $$\left\Vert J^{\otimes k}(x)\right\Vert\exp\bigl(-\langle J(x),{\bm\beta}''\rangle\bigr)\leq f(x)\,.
 $$
The integral
 $$
 \int_M D^k\Bigl(\exp\bigl(-\langle J(x), {\bm\beta}\rangle\bigr)\Bigr)\lambda_\omega(\d x)
 $$
is therefore normally convergent. It follows that the partition function
$P$ is of class $C^\infty$, and that its differentials of all orders can be calculated
by differentiation under the sign $\int$.
\end{proof}

\begin{defi}
\label{GibbsState}
Let ${\bm\beta}\in\Omega$ be a generalized temperature. The statistical state on $M$
whose probability density, with respect to the Liouville measure $\lambda_\omega$, is expressed as
 $$\rho_{\bm\beta}(x)=\frac{1}{P({\bm\beta})}\exp\bigr(-\langle J(x), {\bm\beta}\rangle\bigr)\,,\quad x\in M\,,
 $$  
is called the \emph{Gibbs state} associated to (or indexed by) $\bm\beta$.
\end{defi}

\begin{enonce}[definition]{Gibbs states of subgroups of the Galilei group}
The Galilei group, so named in honour of  the Italian scientist Galileo Galilei
(1564--1642), is the group of symmetries of the mathematical model of
space-time used in classical (non relativistic) mechanics. It is a 
ten-dimensional Lie group diffeomorphic to $\SO(3)\times\RR^7$. The Lie group of 
symmetries of any isolated mechanical system must contain the Galilei group
as a Lie subgroup. In his book \cite{Souriau1969}, Souriau has proven that 
for any mechanical system made of a set of material objects whose total mass is
non-zero, the set of generalized temperatures, for the  action of the Galilei 
group on the manifold of motions, is empty. There is therefore no    Gibbs state
for these systems. However, Gibbs states for subgroups of the Galilei group do exist
and have interesting interpretations in physics and in cosmology \cite{Souriau1974}.
\par\smallskip

The interested reader will find more results abut the Galilei group and its central extension,
Bargmann's group, in 
\cite{deSaxceVallee2010, deSaxceVallee2012, deSaxceVallee2016}. 
\end{enonce}

\begin{prop}
\label{MaximalityEntropy}
For any generalized temperature ${\bm\beta}\in\Omega$, the integral below
 $${\mathcal E}_{\rho_{\bm\beta}}(J)=\frac{1}{P({\bm\beta})}
     \int_M J(x)\exp\bigr(-\langle J(x), {\bm\beta}\rangle\bigr)\lambda_\omega(\d x)
 $$
is convergent. This integral defines the mean value ${\mathcal E}_{\rho_{\bm\beta}}(J)$
of the moment map $J$ in the    Gibbs state indexed by $\bm\beta$. 
Moreover, for any other continuous statistical state with a probability density 
$\rho_1$ with respect to the Liouville measure $\lambda_\omega$, such that 
${\mathcal E}_{\rho_1}(J)$ exists and is equal to ${\mathcal E}_{\rho_{\bm\beta}}(J)$, 
the entropy functional $s$ satisfies the inequality
$s(\rho_1)\leq s(\rho_{\bm\beta})$, and the equality
$s(\rho_1)=s(\rho_{\bm\beta})$ occurs if and only if $\rho_1=\rho_{\bm\beta}$. 
\end{prop}

\begin{proof}
The normal convergence (which implies the usual convergence) of the integral 
which defines ${\mathcal E}_{\rho_{\bm\beta}}(J)$ follows from Proposition 
\ref{DiffPartitionFunction}. Let $\rho_1$ be the probability density, 
with respect to $\lambda_\omega$, of another
continuous statistical state such that ${\mathcal E}_{\rho_1}(J)$ exists and is 
equal to ${\mathcal E}_{\rho_{\bm\beta}}(J)$. The function, defined on $\RR^+$, 
 $$z\mapsto h(z)=\begin{cases}
       z\log\left({\displaystyle\frac{1}{z}}\right)& \text{ if $z>0$}\\
         0& \text{if $z=0$}\end{cases}
 $$
being convex, the straight line tangent to its graph at one of its point 
$\bigl(z_0, h(z_0)\bigr)$ is always above this graph. Therefore, for all 
$z>0$ and $z_0>0$, the following inequality holds:
 $$h(z)\leq h(z_0)-(1+\log z_0)(z-z_0)=z_0 -z(1+\log z_0)\,.
 $$  
With $z=\rho_1(x)$ and $z_0=\rho_{\bm\beta(}x)$, for any $x\in M$, this inequality becomes 
 $$h\bigl(\rho_1(x)\bigr)=\rho_1(x)\log\left(\frac{1}{\rho_1(x)}\right)\leq \rho_{\bm\beta}(x)
    -\bigl(1+\log \rho_{\bm\beta}(x)\bigr)\rho_1(x)\,.
 $$
By integrating on $M$ both sides of the above inequality, we get, since $\rho_{\bm\beta}$ is 
the probability density of the    Gibbs state indexed by $\bm\beta$,
 $$s(\rho_1)\leq 1 -1 -\int_M\rho_1(x)\log \rho_{\bm\beta}(x)\lambda_\omega(\d x)=s(\rho_{\bm\beta})\,.
 $$
We have proven the inequality  $s(\rho_1)\leq s(\rho_{\bm\beta})$. If $\rho_1=\rho_{\bm\beta}$, of course
$s(\rho_1)=s(\rho_{\bm\beta})$. Conversely, let us now assume that $s(\rho_1)=s(\rho_{\bm\beta})$.
The functions $\varphi_1$ and $\varphi$, defined on $M$, whose expressions are
 $$\varphi_1(x)=\rho_1(x)\log\left(\frac{1}{\rho_1(x)}\right)\,,\
    \varphi(x)=\rho_{\bm\beta}(x)-\bigl(1+\log \rho_{\bm\beta}(x)\bigr)\rho_1(x)\,,\ x\in M\,,
 $$
are continuous, except, maybe, the function $\varphi$ at points $x$ where 
$\rho_{\bm\beta}(x)=0$ and $\rho_1(x)\neq 0$. For the Liouville measure 
$\lambda_\omega$, the subset of $M$ made of these points is of measure $0$, since
$\varphi$ is integrable.
The functions $\varphi$ and $\varphi_1$ satisfy the inequality $\varphi_1\leq \varphi$, 
are integrable on $M$ and their integrals are equal. Their difference $\varphi-\varphi_1$,
everywhere $\geq 0$ on $M$ and with an integral equal to $0$, is therefore everywhere equal to $0$.
Therefore, for any $x\in M$,   
 $$\rho_1(x)\log\left(\frac{1}{\rho_1(x)}\right)=\rho_{\bm\beta}(x)-\bigl(1+\log \rho_{\bm\beta}(x)\bigr)\rho_1(x)\,.\eqno{(*)}
 $$     
For any $x\in M$ such that $\rho_1(x)\neq 0$, we can divide both sides of the 
above equality by $\rho_1(x)$. We get
 $$\frac{\rho_{\bm\beta}(x)}{\rho_1(x)} - \log\left(\frac{\rho_{\bm\beta}(x)}{\rho_1(x)}\right)=1\,.
 $$
The function $z\mapsto z - \log z$ reaches its minimum at only one point $z>0$, the point
$z=1$, and its minimum is equal to $1$. So for all $x\in M$ such that $\rho_1(x)>0$,
$\rho_1(x)=\rho_{\bm\beta}(x)$. At points $x\in M$ such that $\rho_1(x)=0$,
equality $(*)$ proves that $\rho_{\bm\beta}(x)=0$. Therefore
$\rho_1=\rho_{\bm\beta}$ everywhere on $M$.   
\end{proof}

\begin{prop} 
We now assume that $\mathfrak g$ is the Lie algebra of a Lie group $G$ which acts
on the symplectic manifold $(M,\omega)$ by a Hamiltonian action $\Phi$, and that $\varphi$
is the Lie algebra action associated to $\Phi$. The Gibbs state 
indexed by any generalized temperature ${\bm\beta}\in\Omega$ is invariant by the 
restriction of the action $\Phi$ to the one-parameter subgroup 
$\bigl\{\exp(\tau {\bm\beta})\bigm|\tau\in\RR\bigr\}$ of $G$.
\end{prop}

\begin{proof} The orbits of the action of this subgroup on $M$ are the integral 
curves of the Hamiltonian vector field which admits the function 
$x\mapsto\bigl\langle J(x),{\bm\beta}\bigr\rangle$ as Hamiltonian. This function
is therefore constant on each orbit of this one-parameter subgroup. The expression of the
probability density $\rho_{\bm\beta}$ of the    Gibbs state indexed by $\bm\beta$
proves that this probability density too is constant on each orbit of the action of
$\bigl\{\exp(\tau{\bm\beta})\bigm|\tau\in\RR\bigr\}$.
\end{proof}

\subsection{Thermodynamic functions}
\label{subsec:3.2}
In this section, the map $\Phi:G\times M\to M$ is a Hamiltonian action of a Lie group $G$ 
on a connected symplectic manifold $(M,\omega)$ and $J:M\to{\mathfrak g}^*$ is a moment map  
of this action. It is assumed that the open subset $\Omega\subset{\mathfrak g}$ of
generalized temperatures is not empty. In \ref{DefinitionsGeneralizedTemperatures},
we have defined  the    partition function $P$,  whose expression is
 $$P({\bm\beta})=\int_M\exp\bigr(-\langle J(x), {\bm \beta}\rangle\bigr)\lambda_\omega(\d x)\,,\quad {\bm\beta}\in\Omega\,.
 $$
On the set $\Omega$ of generalized temperatures,  we define below other 
\emph{thermodynamic functions} whose expressions can be derived from 
that of $P$.

\begin{defis}
\label{DefisValMoyJetEntropie}
Assumptions and notations here are those of \ref{subsec:3.2}.
\par\noindent
{\rm 1.}\quad The \emph{mean value of the moment map} $J$ is the function,
denoted by $E_J$, defined on $\Omega$ and taking its values in the dual vector space
${\mathfrak g}^*$ of the Lie algebra $\mathfrak g$, whose expression is 
  $$    
   E_J({\bm\beta})={\mathcal E}_{\rho_{\bm\beta}}(J)
   =\frac{1}{P({\bm\beta})}\int_M J(x)\exp\bigr(-\langle J(x), 
     {\bm\beta}\rangle\bigr)\lambda_\omega(\d x)\,,\quad{\bm\beta}\in\Omega\,.
  $$
\par\noindent
{\rm 2.}\quad The \emph{entropy function} is the function, denoted by $S$, defined
on $\Omega$ and taking its values in $\RR\cup\{-\infty\}$, which associates, to each 
generalized temperature ${\bm\beta}\in \Omega$, the entropy of the    Gibbs state
indexed by $\bm\beta$:
  $$S({\bm\beta})=s(\rho_{\bm\beta})=\int_M\rho_{\bm\beta}(x)\log\left(\frac{1}{\rho_{\bm\beta}(x)}
       \right)\lambda_\omega(\d x)\,,\quad {\bm\beta}\in\Omega\,.
  $$ 
\end{defis}

\begin{prop}\label{ExpressionsFonctionsThermoGene}
For each generalized temperature ${\bm\beta}\in\Omega$, the values at $\bm\beta$
of the    thermodynamic functions mean value of $J$ and entropy (defined 
in \ref{DefisValMoyJetEntropie}) are given by the formulae
 \begin{align*} 
  E_J({\bm\beta})&=-\frac{1}{P({\bm\beta})}DP({\bm\beta})=-D(\log P)({\bm\beta})\,,\\
  S({\bm\beta})&=\log P({\bm\beta})+\bigl\langle E_J({\bm\beta}),{\bm\beta}\bigr\rangle
       =\log P({\bm\beta})-\bigl\langle D(\log P)({\bm\beta}), {\bm\beta}\bigr\rangle\,.
 \end{align*}
\end{prop}

\begin{proof}
Proposition \ref{DiffPartitionFunction} states that the    partition 
function $P$ is of class $C^\infty$ and that its differentials of all orders 
can be obtained by differentiation under the sign $\int$. Therefore
 $$DP({\bm\beta})=-\int_MJ\exp\bigl(-\langle J(x),{\bm\beta}\rangle\bigr)\lambda_\omega(\d x)
            =-P({\bm\beta})E_J({\bm\beta})\,,
 $$
which proves the indicated expresions of $E_J(\beta)$. Since for each $x\in M$, we have 
$\displaystyle \rho_{\bm\beta}(x)=\frac{\exp\bigl(-\langle J(x),{\bm\beta}\rangle\bigr)}{P({\bm\beta})}$,
 $$\rho_{\bm\beta}(x)\log\frac{1}{\rho_{\bm\beta}(x)}
 =\frac{1}{P({\bm\beta})}\exp\bigl(-\langle J(x),{\bm\beta}\rangle\bigr)
   \bigl(\langle J(x),{\bm\beta}\rangle+\log P({\bm\beta})\bigr)\,.
 $$
By integration over $M$ of both members of this equality with respect to 
$\lambda_\omega$, we obtain the indicated expression of $S({\bm\beta})$.
\end{proof}

\begin{prop}\label{DifferentiellesFonctionsThermo}
The    thermodynamic functions $E_J$ (mean value of $J$) and $S$ (entropy) 
are of class $C^\infty$ on $\Omega$. The first differential of $E_J$ is the function,
defined on $\Omega$ and taking its values in the space of linear applications of
$\mathfrak g$ in its dual vector space ${\mathfrak g}^*$, whose expression is
  \begin{gather*}
   \bigl\langle DE_J({\bm\beta})(X),Y\bigr\rangle
   =-D^2(\log P)({\bm\beta})(X,Y)\\
   =-\frac{1}{P({\bm\beta})}\int_M\bigl\langle J(x)-E_J({\bm\beta}),X\bigr\rangle
                              \bigl\langle J(x)-E_J({\bm\beta}),Y\bigr\rangle
                              \exp\bigl(-\langle J(x),{\bm\beta}\rangle\bigr)\lambda_\omega(\d x)\,,                              
   \end{gather*}
with $X$ and $Y\in{\mathfrak g}$. For each ${\bm\beta}\in\Omega$,
 $DE_J({\bm\beta})$ can be considered as a bilinear, symmetric form on $\mathfrak g$.                              
\par\smallskip

The differential of the entropy function $S$ at each ${\bm\beta}\in\Omega$ 
is an element of ${\mathfrak g}^*$ whose expression is 
 $$\bigl\langle DS({\bm\beta}),X\bigr\rangle
  =\bigl\langle DE_J({\bm\beta})(X),{\bm\beta}\bigr\rangle\,,\quad X\in{\mathfrak g}\,.
 $$ 
\end{prop}

\begin{proof}
According to \ref{ExpressionsFonctionsThermoGene}, for each ${\bm\beta}\in\Omega$,
$E_J({\bm\beta})=-D(\log P)({\bm\beta})$. Therefore, for all $X$ and 
$Y\in{\mathfrak g}$,
 $$\bigl\langle DE_J({\bm\beta})(X),Y\bigr\rangle
   =-D^2(\log P)({\bm\beta})(X,Y)\,,
 $$
which shows that $DE_J({\bm\beta})$ can be considered as a bilinear, symmetric 
form on $\mathfrak g$. Since $DP({\bm\beta})$ can be obtained by differentiation 
under the sign $\int$,
 $$D(\log P)({\bm\beta})=\frac{1}{P({\bm\beta})}DP({\bm\beta})=-\frac{1}{P({\bm\beta})}
     \int_MJ(x)\exp\bigl(-\langle J(x),{\bm\beta}\rangle\bigr)\lambda_\omega(\d x)\,.
 $$ 
By a second differentiation under the sign $\int$, we therefore obtain, for all
$X$ and $Y\in{\mathfrak g}$,
 \begin{gather*}
  D^2(\log P)({\bm\beta})(X,Y)
  =\frac{1}{P({\bm\beta})}\int_M
      \langle J(x),X\rangle\langle J(x),Y\rangle
       \exp\bigl(-\langle J(x),{\bm\beta}\rangle\bigr)\lambda_\omega(\d x)\\
  \quad+\frac{1}{\bigl(P({\bm\beta})\bigr)^2}DP({\bm\beta})(Y)
          \int_M\langle J(x),X\rangle\exp\bigl(-\langle J(x),{\bm\beta}
           \rangle\bigr)\lambda_\omega(\d x)\,.
 \end{gather*}
Let us replace $DP({\bm\beta})(Y)$ 
and $\int_M\langle J(x),X\rangle\exp\bigl(-\langle J(x),{\bm\beta}
           \rangle\bigr)\lambda_\omega(\d x)$, in the right hand side of this 
equality, by their expressions
 \begin{align*}
  DP({\bm\beta})(Y)
  &=-P({\bm\beta})\bigl\langle E_J({\bm\beta}),Y\bigr\rangle\,,\\
  \int_M\langle J(x),X\rangle\exp\bigl(-\langle J(x),{\bm\beta}
           \rangle\bigr)\lambda_\omega(\d x)
  &=P({\bm\beta})\bigl\langle E_J({\bm\beta}, X\bigr\rangle\,.
 \end{align*}
We obtain
 \begin{align*}
  D^2(\log P)({\bm\beta})(X,Y)
  &=\frac{1}{P({\bm\beta})}\int_M
      \langle J(x),X\rangle\langle J(x),Y\rangle
       \exp\bigl(-\langle J(x),{\bm\beta}\rangle\bigr)\lambda_\omega(\d x)\\
  &\quad-\bigl\langle E_J({\bm\beta}),Y\bigr\rangle
          \bigl\langle E_J({\bm\beta}, X\bigr\rangle\\
  &=\frac{1}{P({\bm\beta})}\int_M
     \bigl\langle J(x)-E_J({\bm\beta}),X\bigr\rangle\bigl\langle J(x)-E_J({\bm\beta}),Y\bigr\rangle\\
  &\quad\quad\quad\quad\exp\bigl(-\langle J(x),{\bm\beta}\rangle\bigr)\lambda_\omega(\d x)\,,     
 \end{align*}           
The expression of $\bigl\langle DE_J({\bm\beta})(X),Y\bigr\rangle$ given 
in the statement follows.
\par\smallskip

By differentiation of the expression of $S({\bm\beta})$ given in
\ref{ExpressionsFonctionsThermoGene}, and using the equality
$DP({\bm\beta})=-P({\bm\beta})E_J({\bm\beta})$, we obtain, for any $X\in{\mathfrak g}$,
 \begin{align*}
  \bigl\langle DS({\bm\beta}),X\bigr\rangle
  &=\frac{1}{P({\bm\beta})}DP({\bm\beta})(X)+
       \bigl\langle E_J({\bm\beta}),X\bigr\rangle+\bigl\langle DE_J({\bm\beta})(X),{\bm\beta}\bigr\rangle\\
  &=\bigl\langle DE_J({\bm\beta})(X),{\bm\beta}\bigr\rangle\,.\qedhere
 \end{align*}        
\end{proof}

\begin{theo}\label{EJDiffeo}
For all ${\bm\beta}\in \Omega$, $X$ and $Y\in{\mathfrak g}$, let
 $$\Gamma({\bm\beta})(X,Y)=-\bigl\langle DE_J({\bm\beta})(X),Y\bigr\rangle 
                    =D^2(\log P)({\bm\beta})(X,Y)\,.
 $$
The map $\Gamma$ so defined is a $C^\infty$ bilinear, symmetric differential form defined on
$\Omega$ such that, for each ${\bm\beta}\in\Omega$ and $X\in{\mathfrak g}$,
 $$\Gamma({\bm\beta})(X,X)\geq 0\,.
 $$
Moreover, if $X\in{\mathfrak g}$ is such that $x\mapsto \bigl\langle J(x),X\bigr\rangle$ 
is not a constant function,
 $$\Gamma({\bm\beta})(X,X)>0\,.
 $$
When, in addition, the Hamiltonian
action $\Phi:G\times M\to M$ is effective (it means that for any $X\in{\mathfrak g}$,
$X\neq 0$, the function $x\mapsto  \bigl\langle J(x),X\bigr\rangle$ is not a constant on $M$), 
$\Gamma$ is a Riemannian metric on $\Omega$. Moreover, the map
$E_J:\Omega\to{\mathfrak g}^*$ is injective, its image is an open subset $\Omega^*$
of ${\mathfrak g}^*$, and considered as valued in $\Omega^*$, $E_J$ is a diffeomorphism 
of the set $\Omega$ of generalized temperatures onto the open subset $\Omega^*$ of ${\mathfrak g}^*$.                          
\end{theo}

\begin{proof}
The firt assertions follow from the the expression of $\bigl\langle DE_J({\bm\beta})(X),Y\bigr\rangle$ 
given in Proposition \ref{DifferentiellesFonctionsThermo}.
When $X\in{\mathfrak g}$ is such that the function $x\mapsto \bigl\langle J(x),X\bigr\rangle$
is not a constant on $M$, the function 
 $$x\mapsto \bigl\langle J(x)-E_J({\bm\beta}), X\bigr\rangle^2\exp\bigl(-\langle J,{\bm\beta}\rangle\bigr)
 $$
is continuous, with values $\geq 0$ and not everywhere equal to $0$ on $M$. 
Its integral with respect to the Liouville measure is therefore strictly positive, 
which proves that $\Gamma({\bm\beta})(X,X)>0$. 
\par\smallskip

When, in addition, the action $\Phi$ effective, for any ${\bm\beta}\in \Omega$ 
and any non-zero  $X\in{\mathfrak g}$, $\Gamma({\bm\beta})(X,X)>0$. The map $\Gamma$
is therefore an Riemannian metric on $\Omega$. For all
${\bm\beta}\in\Omega$ and $Y\in{\mathfrak g}\backslash\{0\}$, we have 
$\bigl\langle DE_J({\bm\beta})(Y),Y\bigr\rangle<0$, which implies that 
$DE_J({\bm\beta})$ is invertible. The map $E_J:\Omega\to{\mathfrak g}^*$ is therefore
open. This map cannot take the same value at two distinct points  
${\bm\beta}_1$ an ${\bm\beta}_2\in\Omega$, since this would imply
 $$\bigl\langle E_J({\bm\beta}_1),{\bm\beta}_2-{\bm\beta}_1\bigr\rangle=
    \bigl\langle E_J({\bm\beta}_2),{\bm\beta}_2-{\bm\beta}_1\bigr\rangle\,.
 $$
The real-valued function
 $$\lambda\mapsto \Bigl\langle E_J\bigl((1-\lambda){\bm\beta}_1+\lambda{\bm\beta}_2\bigr), 
    {\bm\beta}_2-{\bm\beta}_1\Bigr\rangle\,,\quad\lambda\in[0,1]\,,
 $$
would be well defined on $[0,1]$ since $\Omega$ is convex, smooth in $]0,1[$, and would take the
same value for $\lambda=0$ and $\lambda=1$.
Its derivative with respect to $\lambda$, whose value is
$\Bigl\langle DE_J\bigl(\lambda {\bm\beta}_1
  +(1-\lambda){\bm\beta}_2\bigr)({\bm\beta}_2-{\bm\beta}_1),({\bm\beta}_2-{\bm\beta}_1)
   \Bigr\rangle$
would vanish fore some $\lambda\in]0,1[$,
which would contradict the effectiveness of $\Phi$. Being open and injective,
the map $E_J:\Omega\to{\mathfrak g}^*$ is a diffeomorphism of $\Omega$ onto its image
$\Omega^*$, which is an open subset of ${\mathfrak g}^*$.
\end{proof}

\begin{remas}\label{EJDiffeoRmks} Theorem \ref{EJDiffeo} leads to the following observations.

\par\smallskip\noindent
{\rm 1.}\quad In the language of Probability theory,
$-\bigl\langle DE_J({\bm\beta})(X),X\bigr\rangle$ is the \emph{variance}, in other words
the square of the \emph{standard deviation} of the random variable $\langle J,X\rangle$,
for the probability law $\rho_{\bm\beta}\lambda_\omega$ on $M$. 

\par\smallskip\noindent
{\rm 2.}\quad For each generalized temperature ${\bm\beta}\in\Omega$, the    Gibbs state 
indexed by ${\bm\beta}$ is the probability law on $M$, absolutely continuous with respect
to the Liouville measure $\lambda_\omega$, of probability density
 $$\rho_{\bm\beta}=\frac{1}{P({\bm\beta})}\exp\bigl(-\langle J,{\bm\beta}\rangle\bigr)\,.
 $$
The open subset $\Omega$ of $\mathfrak g$, in which live the generalized temperatures 
${\bm\beta}$ which index a familly of probability laws defined on $M$, 
is called by statisticians a \emph{statistical manifold}. The \emph{Fisher-Rao metric},
so named in honour of the British statistician and genetician Ronald Aylmer Fisher (1890--1962)
and the Indian statistician Calyampudi Radhakrishna Rao (born in 1920, emeritus 
Professor at the Indian Statistics Institute and at the Pennsylvania State University)
is a Riemannian metric, defined on some statistical manifolds, which is used to
evaluate the distance between probability laws. Frédéric Barbaresco
\cite{Barbaresco2016} observed that the Riemannian metric $\Gamma$ defined by
Jean-Marie Souriau on $\Omega$ is nothing else than the Fisher-Rao metric when, 
as indicated above, $\Omega$ is considered as a statistical manifold.

\par\smallskip\noindent
{\rm 3.}\quad Under the assumptions of \ref{EJDiffeo}, the equality
 $$S({\bm\beta})=\bigl\langle D(-\log P)({\bm\beta}), {\bm\beta}\bigr\rangle-(-\log P)({\bm\beta})
 $$
proves that each of the two functions  $-\log P:\Omega\to\RR$ and 
$S\circ E_J^{-1}:\Omega^*\to\RR$ is the \emph{Legendre transform}
of the other, just as a hyper-regular Lagrangian $L:TM\to\RR$ and the associated
Hamiltonian $H:T^*M\to\RR$, defined, respectively, on the tangent bundle $TM$ 
and on the cotangent bundle $T^*M$ to some smooth manifold $M$. Here the Legendre map
is $E_J:\Omega\to{\Omega}^*$. This map and its inverse $(E_J)^{-1}:\Omega^*\to\Omega$ 
are expressed by formulae similar to those which express the Legendre map
$TM\to T^*M$ and its inverse $T^*M\to TM$ in calculus of variations,
 $$E_J=D(-\log P)\,,\quad (E_J)^{-1}=D(S\circ {E_J}^{-1})\,.
 $$  
\end{remas} 

The moment map $J$ of the Hamiltonian action $\Phi$ is not unique: it is well 
known that for any constant $\mu\in{\mathfrak g}$, $J+\mu$ is too a moment map of 
$\Phi$. Proposition \ref{EffetChangementJ} below indicates the effet of such a change 
on the    thermodynamic functions $P$, $E_J$ and $S$.

\begin{prop}\label{EffetChangementJ}
Let  $\mu\in{\mathfrak g}^*$ be a constant. When the moment map $J$ of the 
Hamiltonian action $\Phi$ is replaced by $J_1=J+\mu$, the set $\Omega$ of    
generalized temperatures does not change. The    thermodynamic functions
$P$, $E_J$ and $S$ are replaced, respectively, by $P_1$, $E_{J_1}$ and $S_1$, whose expressions are
 $$P_1({\bm\beta})=\exp\bigl(-\langle \mu,{\bm\beta}\rangle\bigr)P({\bm\beta})\,,\quad 
   E_{J_1}({\bm\beta})=E_J({\bm\beta})+\mu\,,\quad  S_1({\bm\beta})=S({\bm\beta})\,.
 $$
For each ${\bm\beta}\in\Omega$, the associated    Gibbs state, its probability
density $\rho_{\bm\beta}$ with respect to the Liouville measure $\lambda_\omega$ and the
bilinear, symmetric form $\Gamma$ (theorem \ref{EJDiffeo}) are not changed.
\end{prop}

\begin{proof} The stated results follow from the equality
 $$\exp\bigl(-\langle J+\mu, {\bm\beta}\rangle\bigr)=
    \exp\bigl(-\langle\mu,{\bm\beta}\rangle\bigr)
     \exp\bigl(-\langle J,{\bm\beta}\rangle\bigr)\,.\eqno{\qedhere}
 $$
\end{proof}

\subsection{Generalized temperatures and adjoint action}
\label{subsec:3.3}

As in the previous section, $\Phi:G\times M\to M$ is a Hamiltonian action of a connected Lie
group $G$ on a connected symplectic manifold $(M,\omega)$ and $J:M\to{\mathfrak g}^*$
is a moment map of this action. The set of generalized temperatures is assumed to be
a non-empty subset $\Omega$ of the Lie algebra $\mathfrak g$. As seen in 
\ref{EffetChangementJ}, $\Omega$ does not depend on the choice of the moment map $J$.
We moreover assume that $\Phi$ is effective, which implies (theorem \ref{EJDiffeo})
that $E_J$ is a diffeomorphism of $\Omega$ onto an open subset $\Omega^*$ of 
${\mathfrak g}^*$, and that the bilinear, symmetric form 
$\Gamma$ is a Riemannian metric on $\Omega$. 
By considering the adjoint action of $G$ on $\Omega$, we prove below that $\Omega$ is a union of 
adjoint orbits (proposition \ref{ActionAdjointeSurPetEJetS}) and that the 
Riemannian metric induced by $\Gamma$ on each of 
these orbits can be expressed in terms of a symplectic cocycle of the 
Lie algebra $\mathfrak g$ (theorem \ref{ExpressionGammaAdjointe}).
\par\smallskip

The next proposition proves that $\Omega$ is a union of adjoint orbits and indicates
the variations of the    thermodynamic functions
$P$, $E_J$ and $S$ on each adjoint orbit contained in $\Omega$.

\begin{prop}\label{ActionAdjointeSurPetEJetS}
The set $\Omega$ of generalized temperatures is a union of orbits of the adjoint 
action of the Lie group $G$ on its Lie algebra $\mathfrak g$. Let 
$\theta:G\to{\mathfrak g}^*$ be the symplectic cocycle of $G$ 
(see, for example, \cite{Marle2018}) such that, for each $g\in G$
 $$ J\circ\Phi_g=\Ad^*_{g^{-1}}\circ\, J+\theta(g)\,.
 $$ 
For any ${\bm\beta}\in\Omega$ and any $g\in G$, we have
 \begin{align*} 
       P(\Ad_g{\bm\beta})  &=\exp\Bigl(\bigl\langle\theta(g^{-1}),{\bm\beta}\bigr\rangle\Bigr)P({\bm\beta})
                   =\exp\Bigl(-\bigl\langle\Ad^*_g\theta(g),{\bm\beta}\bigr\rangle\Bigr)P({\bm\beta})\,,\\
       E_J(\Ad_g{\bm\beta})&=\Ad^*_{g^{-1}}E_J({\bm\beta})+\theta(g)\,,\\
       S(\Ad_g{\bm\beta})  &=S({\bm\beta})\,. 
 \end{align*}
\end{prop}

\begin{proof}
Let us assume that the integral which defines $P(\Ad_g{\bm\beta})$ is 
convergent. We can write
 \begin{align*}
  P(\Ad_g{\bm\beta})&=\int_M\exp\bigl(-\langle J(x),\Ad_g{\bm\beta}\rangle\bigr)\lambda_\omega(\d x)\\
            &=\int_M\exp\bigl(-\langle \Ad^*_g J(x),{\bm\beta}\rangle\bigr)\lambda_\omega(\d x)\\
           &=\int_M\exp\Bigl(-\bigl\langle J\circ\Phi_{g^{-1}}(x)-\theta(g^{-1}),{\bm\beta}\bigr\rangle\Bigr)
              \lambda_\omega(\d x)\\
           &=\exp\Bigl(\bigl\langle\theta(g^{-1}),{\bm\beta}\bigr\rangle\Bigr)
              \int_M\exp\Bigl(-\bigl\langle J\circ\Phi_{g^{-1}}(x),{\bm\beta}\bigr\rangle\Bigr)
              \lambda_\omega(\d x)\,.           
 \end{align*}
The change of integration variable $y=\Phi_{g^{-1}}(x)$ in the last integral leads to
 $$\int_M\exp\Bigl(-\bigl\langle J\circ\Phi_{g^{-1}}(x),{\bm\beta}\bigr\rangle\Bigr)
              \lambda_\omega(\d x)
 =\int_M\exp\Bigl(-\bigl\langle J(y),{\bm\beta}\bigr\rangle\Bigr)
              \Phi_g^*\lambda_\omega(\d y)=P({\bm\beta})\,,
 $$
since $\Phi_g^*\lambda_\omega=\lambda_\omega$, the Liouville measure being invariant
by symplectomophisms. Moreover, $\theta(g^{-1})=-\Ad^*_g\theta(g)$ (see for example \cite{Marle2018}), so we can write
 $$P(\Ad_g{\bm\beta})=\exp\Bigl(-\bigl\langle\Ad^*_g\theta(g),{\bm\beta}\bigr\rangle\Bigr)P({\bm\beta})\,.
 $$
By reversing the above calculation step by step, we prove that the normal 
convergence of the integral which defines $P({\bm\beta})$ implies the normal 
convergence of the integral which defines $P(\Ad_g{\bm\beta})$. We therefore 
have proven that $\Omega$ is a union of adjoint orbits of $G$, as well as the
expression of $P(\Ad_g{\bm\beta})$ in terms of $P({\bm\beta})$ and $\theta$ 
given in the statement.
\par\smallskip

Since $E_J({\bm\beta})=-D(\log P)({\bm\beta})$, $E_J(\Ad_g{\bm\beta})=-D(\log P)(\Ad_g{\bm\beta})$.
To calculate the right hand side of this equality, we observe that for each
${\bm \delta}\in{\mathfrak g}$ and each real $s$,
 \begin{align*}
  D(\log P)(\Ad_g{\bm\beta})({\bm\delta})
  &=\frac{\d}{\d s}\bigl(\log P(\Ad_g{\bm\beta}+s{\bm\delta})\bigr)\bigm|_{s=0}\\
  &=\frac{\d}{\d s}\Bigl(\log P\bigl(\Ad_g({\bm\beta}+s\Ad_{g^{-1}}{\bm\delta}\bigr)\Bigr)\bigm|_{s=0}
 \end{align*}
Using the expression of $P(\Ad_g{\bm\beta})$ obtained above, we have
 $$\log P\bigl(\Ad_g({\bm\beta}+s\Ad_{g^{-1}}{\bm\delta})\bigr)
  =-\langle\Ad^*_g\theta(g),{\bm\beta}+s\Ad_{g^{-1}}{\bm\delta}\rangle+\log P({\bm\beta}+s\Ad_{g^{-1}}{\bm\delta})\,.
 $$
Taking the derivative with respect to $s$, then  setting $s=0$, we get
 \begin{align*}
 D\log P(\Ad_g{\bm\beta})({\bm\delta})
 &=-\bigl\langle\Ad^*_g\theta(g),\Ad_{g^{-1}}{\bm\delta}\bigr\rangle
   +D\log P({\bm\beta})(\Ad_{g^{-1}}{\bm\delta})\\
 &=-\bigl\langle\theta(g),{\bm\delta}\bigr\rangle+D\log(P)({\bm\beta})(\Ad_{g^{-1}}{\bm\delta})\\
 &=-\bigl\langle\theta(g)+\Ad^*_{g^{-1}}E_J({\bm\beta}),{\bm\delta}\bigr\rangle\,,
 \end{align*}
where we have used the already obtained equality $D\log P({\bm\beta})=-E_J({\bm\beta})$.
Therefore,
 $$E_J(\Ad_g{\bm\beta})=\Ad^*_{g^{-1}}E_J({\bm\beta})+\theta(g)\,.
 $$
Finally,
 \begin{align*}
 S(\Ad_g{\bm\beta})
  &=\log P(\Ad_g{\bm\beta})-\bigl\langle D\log P(\Ad_g{\bm\beta}),\Ad_g{\bm\beta}\bigr\rangle\\
  &=-\bigl\langle\Ad^*_g\theta(g),{\bm\beta}\bigr\rangle+\log P({\bm\beta})
     +\bigl\langle \Ad^*_{g^{-1}}E_J({\bm\beta})+\theta(g),\Ad_g{\bm\beta}\bigr\rangle\\
  &=\log P({\bm\beta})+\bigl\langle E_J({\bm\beta}),{\bm\beta}\bigr\rangle\\
  &=S({\bm\beta})\,.\qedhere
 \end{align*}                     
\end{proof} 

\begin{rema}
The equality
 $$ E_J(\Ad_g\bm\beta)=\Ad^*_{g^{-1}}E_J({\bm\beta})+\theta(g)
 $$
states that the map $E_J:\Omega\to\Omega^*$ is equivariant with respect to the
adjoint action $\Phi$ of $G$ on $\mathfrak g$, restricted to the open subset 
$\Omega$ of $\mathfrak g$, and its 
affine action $a_\theta$ on ${\mathfrak g}^*$:
 $$a_\theta (g,\xi)=\Ad^*_{g^{-1}}\xi+\theta(g)\,,\quad g\in G\,,\quad \xi\in{\mathfrak g}^*\,,
 $$
restricted to the open subset $\Omega^*$ of ${\mathfrak g}^*$.
This result is not surprising, since it is well known (see, for example, \cite{Marle2018}) 
that the  moment map $J$ itself is equivariant with respect to the action $\Phi$ 
of $G$ on $M$ and its affine action $a_\theta$ on ${\mathfrak g}^*$: it states 
that the equivariance of $J$ implies the equivariance of its mean value. 
\end{rema} 

\begin{prop}\label{FormulesEJetTheta}
Let $\Theta=T_e\theta:{\mathfrak g}\to{\mathfrak g}^*$ be the $1$-cocycle of the 
Lie algebra ${\mathfrak g}$ associed to the symplectic $1$-cocycle $\theta$ of the 
Lie group $G$ (see, for example, \cite{Marle2018}). For any
${\bm\beta}\in\Omega$ and any $X\in{\mathfrak g}$, 
   \begin{align*}
    \bigl\langle E_J({\bm\beta}),[X,{\bm\beta}]\bigr\rangle&=\bigl\langle\Theta(X),{\bm\beta}\bigr\rangle\,,\\
    DE_J({\bm\beta})\bigl([X,{\bm\beta}]\bigr)&=-\ad^*_XE_J({\bm\beta}) + \Theta(X)\,.
   \end{align*}
\end{prop}

\begin{proof}
Let us set $g=\exp(\tau X)$ in the expression of $P(\Ad_g{\bm\beta})$ given in
proposition \ref{ActionAdjointeSurPetEJetS}, then take the derivative with respect to 
$\tau$ and set $\tau=0$. Using the well known equalities 
$\theta(e)=0$ and $T_e\theta=\Theta$, we obtain
 \begin{align*}
 DP({\bm\beta})\bigl([X,b]\bigr)
 &=\frac{\d}{\d\tau}
    \Bigl(\exp\bigl(-\bigl\langle \Ad^*_{\exp(\tau X)}\theta\bigl(\exp(\tau X)
     \bigr),{\bm\beta}\bigr\rangle\bigr) P({\bm\beta})\Bigr)\Bigm|_{\tau=0}\\
 &=-\bigl\langle\Theta(X),{\bm\beta}\bigr\rangle P({\bm\beta})\,,
 \end{align*}
which proves the first assertion, since $DP({\bm\beta})=-P({\bm\beta})E_J({\bm\beta})$.
\par\smallskip

Similarly, let us set $g=\exp(\tau X)$ in the expression of
$E_J(\Ad_g{\bm\beta})$ given in proposition \ref{ActionAdjointeSurPetEJetS}, 
then take the derivative with respect to $\tau$ and set $\tau=0$. We obtain
 \begin{align*}
  DE_J({\bm\beta})\bigl([X,{\bm\beta}]\bigr)
  &=\frac{\d}{\d\tau}\Bigl(\Ad^*_{\exp(-\tau X)}E_J({\bm\beta})
    +\theta\bigl(\exp(\tau X)\bigr)\Bigr)\Bigm|_{\tau=0}\\
  &=-\ad^*_XE_J({\bm\beta})+\Theta(X)\,.\qedhere
 \end{align*}
\end{proof}

\begin{theo}\label{ExpressionGammaAdjointe} Let us set,
for each generalized temperature ${\bm\beta}\in\Omega$, 
 $$J_{\bm\beta}=J-E_J({\bm\beta})\,,$$
and, for each $g\in G$,
 $$\theta_{\bm\beta}(g)=\theta(g)-E_J({\bm\beta})+\Ad^*_{g^{-1}}E_J({\bm\beta})\,.
 $$
The map $J_{\bm\beta}$ is the unique moment map of the Hamiltonian action $\Phi$ 
whose mean value, for the generalized temperature $\bm\beta$, is equal to $0$. 
The map $\theta_{\bm\beta}$ is the symplectic cocycle of the Lie group $G$, 
cohomologous to $\theta$, associated to the moment map $J_{\bm\beta}$. It depends 
on ${\bm\beta}$ but not on the choice of  $J$.
\par\smallskip

Let $\Theta_{\bm\beta}:{\mathfrak g}\to{\mathfrak g}^*$ be the symplectic cocycle
of the Lie algebra ${\mathfrak g}$ associated to the symplectic $1$-cocycle 
$\theta_{\bm\beta}$ of the Lie group $G$ (see, for example, \cite{Marle2018}). 
Its expression is
 $$\Theta_{\bm\beta}(X)=T_e\theta_{\bm\beta}(X)=\Theta(X)-\ad^*_XE_J({\bm\beta})\,.
 $$
The map $\Theta_{\bm\beta}$ is the unique symplectic $1$-cocycle of the Lie algebra
$\mathfrak g$ which is cohomologous to $\Theta$ and satisfies the equality
 $$\Theta_{\bm\beta}({\bm\beta})=0\,.
 $$
\par\smallskip
Let $X$ and $Y$ be two elements in ${\mathfrak g}$, considered as two elements of
$T_{\bm\beta}\Omega$, in other words as two vectors tangent 
to $\Omega$ at its point $\bm\beta$. Let us moreover assume that $X$ is tangent 
to the adjoint orbit of $\bm\beta$ at its point $\bm\beta$. There exists 
$X_1\in{\mathfrak g}$ such that $X=[{\bm\beta},X_1]$. When evaluated on the pair
of tangent vectors $(X,Y)$, the Riemannian metric $\Gamma$ can be expressed as
 $$\Gamma({\bm\beta})(X,Y)=\bigl\langle\Theta_{\bm\beta}(X_1),Y\bigr\rangle\,.                  
 $$
If $Y$ too is tangent to the adjoint orbit of $\bm\beta$ at its point $\bm\beta$,
there exists $Y_1\in{\mathfrak g}$ such that $Y=[{\bm\beta},Y_1]$, and we have the two
equalities, which express the Riemannian metric induced by $\Gamma$ on the adjoint orbit of
$\bm\beta$,
 $$\Gamma({\bm\beta})(X,Y)=\bigl\langle\Theta_{\bm\beta}(X_1),Y\bigr\rangle
                      =\bigl\langle\Theta_{\bm\beta}(Y_1),X\bigr\rangle\,.                  
 $$                      
\end{theo}

\begin{proof}
Since $\Theta$, being a symplectic cocycle, is skew-symmetric, we have, for each 
$X\in{\mathfrak g}$, 
$\bigl\langle \Theta({\bm\beta}),X\bigr\rangle=-\bigl\langle \Theta(X),{\bm\beta}\bigr\rangle$.
Using the equalities proven in \ref{FormulesEJetTheta}, we obtain
 \begin{align*}
   \bigl\langle \Theta_{\bm\beta}({\bm\beta}),X\bigr\rangle
   &=\bigl\langle \Theta({\bm\beta}),X\bigr\rangle -\bigl\langle\ad^*_{\bm\beta} E_J({\bm\beta}),X\bigr\rangle
    =-\bigl\langle\Theta(X),{\bm\beta}\bigr\rangle -\bigl\langle E_J({\bm\beta}), [{\bm\beta},X]\bigr\rangle\\
   &=-\bigl\langle E_J({\bm\beta}),[X,{\bm\beta}]\bigr\rangle-\bigl\langle E_J({\bm\beta}),[{\bm\beta},X]\bigr\rangle
    =0\,.
 \end{align*}
Other statements about $J_{\bm\beta}$, $\theta_{\bm\beta}$ and $\Theta_{\bm\beta}$
easily follow from well known properties of moment maps of Hamiltonian actions 
(see for example \cite{Marle2018}). 
\par\smallskip

Using theorem \ref{EJDiffeo} and proposition \ref{FormulesEJetTheta}, we obtain, 
for all ${\bm\beta}\in\Omega$, $X_1$ and $Y\in{\mathfrak g}$, with
$X=[X_1,{\bm\beta}]$,
 \begin{align*}
  \Gamma({\bm\beta})\bigl([X_1,{\bm\beta}],Y\bigr)
   =-\bigl\langle DE_J({\bm\beta})\bigl([X_1,{\bm\beta}]\bigr),Y\bigr\rangle
   =\bigl\langle\ad^*_{[X_1,{\bm\beta}]}E_J({\bm\beta})+\Theta(X_1),Y\bigr\rangle\,.
 \end{align*}
According to proposition \ref{EffetChangementJ}, the bilinear form $\Gamma$
does not depend on the choice of the moment map $J$, so we can replace $J$ by 
$J_{\bm\beta}$ in the right hand side of the above equality. Of course we have to replace
too $E_J$ by $E_{J_{\bm\beta}}$ and $\Theta$ by $\Theta_{\bm\beta}$. The map
$J_{\bm\beta}$ was chosen so that$ E_{J_{\bm\beta}}({\bm\beta})=0$, so we obtain
 $$\Gamma({\bm\beta})\bigl([X_1,{\bm\beta}],Y\bigr)
   =\bigl\langle\Theta_{\bm\beta}(X_1),Y\bigr\rangle\,.
 $$
When we both have $X=[X_1,{\bm\beta}]$ and $Y=[Y_1,{\bm\beta}]$, with $X_1$ and 
$Y_1\in{\mathfrak g}$, we can exchange the parts played by 
$X$ and $Y$ and write
 \begin{align*}
  \Gamma({\bm\beta})(X,Y)
   &=\Gamma({\bm\beta})\bigl([X_1,{\bm\beta}],[Y_1,{\bm\beta}]\bigr)
    =\Gamma({\bm\beta})\bigl([Y_1,{\bm\beta}],[X_1,{\bm\beta}]\bigr)\\
   &=\Gamma({\bm\beta})\bigl([Y_1,{\bm\beta}],X\bigr) 
    =\bigl\langle\Theta_{\bm\beta}(Y_1),X\bigr\rangle\,.\qedhere
 \end{align*}
\end{proof}

\section{Examples of    Gibbs states}
\label{sec:4}

This section describes several examples of    Gibbs states and the 
associated    thermodynamic functions. It begins with a subsection (\ref{subsec:4.1}) in which 
some properties of oriented three-dimensional Euclidean or pseudo-Euclidean vector spaces are recalled. 
A remarkable isomorphism of such a vector space onto the Lie algebra
of its group of symmetries is defined (\ref{subsub:RemarkableIsomorphism}). This isomorphism, which can be expressed
in terms of the Hodge star operator (\ref{subsub:ExpressionHodge}), is well known 
and often used in mechanics when the considered vector space is properly Euclidean, 
maybe a little less well known when it is pseudo-Euclidean. With its use, the considered vector space
can be endowed both with a Lie algebra structure and a Lie-Poisson structure (\ref{subsub:LieAlgebraAndPoissonOnF}).
The coadjoint orbits of its group of symmetries can be considered as submanifolds of this vector space
(\ref{subsub:CoadjointOrbitsInF}), a property used 
in the following subsection (\ref{subsec:4.2}) for the determination of    Gibbs states 
on several two-dimensional symplectic manifolds: the two-dimensional sphere (\ref{subsub:4.2.1}), 
the two-dimensional pseudo-sphere (\ref{subsub:4.2.2}),
the Poincaré disk (\ref{subsub:4.2.3}) and the Poincaré half-plane (\ref{subsub:4.2.4}). In  
\ref{subsub:4.2.5}, it is proven that on a two-dimensional symplectic vector space, 
there is no    Gibbs state for the action of the linear symplectic group. 
Finally, the    Gibbs states and the associated thermodynamic functions 
for the action, on an Euclidean affine space, of the group of its displacements, 
is determined in \ref{subsub:4.2.6}.

\subsection{Three-dimensional oriented vector spaces}
\label{subsec:4.1}
In what follows, $\zeta$ is a real integer whose value is either $+1$ or $-1$, and
${\F}$ is a three-dimensional real vector space endowed with a scalar product 
$\F\times\F\to\RR$, denoted by $({\bf v},{\bf w})\mapsto {\bf v}\cdot {\bf w}$, 
with ${\bf v}$ and ${\bf w}\in\F$, whose signature is $(+,+,+)$ when $\zeta=1$ 
and $(+,+,-)$ when $\zeta=-1$. This scalar product is Euclidean when $\zeta=1$ 
and pseudo-Euclidean when $\zeta=-1$. A  basis $({\bf e}_x,{\bf e}_y,{\bf e}_z)$ 
of $\F$ is said to be \emph{orthonormal} when 
 $${\bf e}_x\cdot{\bf e}_x={\bf e}_y\cdot{\bf e}_y=1\,,\
 {\bf e}_z\cdot{\bf e}_z=\zeta\,,\ 
 {\bf e}_x\cdot{\bf e}_y={\bf e}_y\cdot{\bf e}_z={\bf e}_z\cdot{\bf e}_x=0\,.
 $$
When $\zeta=-1$, the vector space $\F$ is called a \emph{three-dimensional Minkowski
vector space}. A non-zero element ${\bf v}\in\F$ is said to be \emph{space-like}
when ${\bf v}\cdot{\bf v}>0$, \emph{time-like} when ${\bf v}\cdot{\bf v}<0$ and
\emph{light-like} when ${\bf v}\cdot{\bf v}=0$. The subset of $\F$ made of non-zero
time-like or light-like elements has two connected components. A \emph{temporal orientation}
of $\F$ is the choice of one of these two connected components, whose elements are said
to be \emph{directed towards the future}. 
\par\smallskip

Both when $\zeta=1$ and when $\zeta=-1$, we will assume in what follows that an 
orientation of $\F$ in the usual sense is chosen, and when $\zeta=-1$, we will assume
that a temporal orientation of $\F$ is chosen too. The orthonormal bases of $\F$ 
used will always be chosen positively oriented and, when $\zeta=-1$, their third 
element ${\bf e}_z$ will be chosen time-like and directed towards the future. Such bases 
of $\F$ will be called \emph{admissible bases}.

We denote by $G$ be the subset of $\GL(\F)$ made of linear automorphisms $g$ of $\F$ which transform
any admissible basis $({\bf e}_x,{\bf e}_y,{\bf e}_z)$ of $\F$ 
into an admissible basis $\bigl(g({\bf e}_x),g({\bf e}_y),g({\bf e}_z)\bigr)$. Elements $g$ of $G$ 
preserve the scalar product in $\F$, \emph{i.e.}, 
they are such that, for any pair $({\bf v},{\bf w})\in\F\times\F$,
 $$g({\bf v})\cdot g({\bf w})={\bf v}\cdot{\bf w}\,.
 $$  
Moreover, they preserve the orientation of $\F$ and, when $\zeta=-1$, its temporal
orientation. The subset $G$ of $\GL(\F)$ is the group of symmetries of $\F$, 
endowed with its scalar product, its orientation and, when $\zeta=-1$, its temporal
orientation. It is a connected Lie group isomorphic to the rotation
group $\SO(3)$ when $\zeta=1$, and to the restricted three-dimensional Lorentz group $\SO(2,1)$ 
when $\zeta=-1$. Its Lie algebra, which will be denoted by $\mathfrak g$, is 
therefore isomorphic to $\so(3)$ when $\zeta=1$, and to $\so(2,1)$ when $\zeta=-1$.
\par\smallskip

Some useful properties of the vector space $\F$, of its symmetry group $G$ and of the
Lie algebra $\mathfrak g$ are recalled below. The interested reader will 
find their detailed proofs in \cite{Marle2019} or, for most of them, 
in the very nice book \cite{NovikovTaimanov}.  

\begin{enonce}[definition]{A remarkable Lie algebras isomorphism}
\label{subsub:RemarkableIsomorphism}   
Let $({\bf e}_x,{\bf e}_y,{\bf e}_z)$ be an admissible basis of $\F$, in the sense indicated
in \ref{subsec:4.1}. 
For any triple $(a,b,c)\in\RR^3$, let $j(a{\bf e}_x+b{\bf e_y}+c{\bf e}_z)$ be
the linear endomorphism of $\F$ whose matrix, in the basis $({\bf e}_x,{\bf e}_y,{\bf e}_z)$, is
 $$\hbox{matrix of\ }j(a{\bf e}_x+b{\bf e}_y+c{\bf e}_z)=
    \begin{pmatrix}0&-c&b\cr
                   c&0&-a\cr
                   -\zeta b&\zeta a&0\cr
    \end{pmatrix}\,.
 $$
The map $j$ does not depend on the admissible basis of $\F$ used for its definition. 
This property follows from the fact that $j$ can be expressed in terms of the Hodge star
operator, as explained below in \ref{subsub:ExpressionHodge}. It is linear and injective, 
and its image is the Lie algebra $\mathfrak g$, considered
as a vector subspace of the vector space ${\mathcal L}(\F,\F)$ of linear endomorphisms of $\F$.
There exists a unique  bilinear and skew-symmetric map, defined on $\F\times\F$ and with values in $\F$,
denoted by $({\bf v},{\bf w})\mapsto {\bf v}\dot\times{\bf w}$, such that, for all ${\bf v}$ and
${\bf w}\in\F$,
 $$j({\bf v}\dot\times{\bf w})=\bigl[j({\bf v}),j({\bf w})\bigr]
                              =j({\bf v})\circ j({\bf w})-j({\bf w)}\circ j({\bf v})\,. 
 $$
The bilinear map $({\bf v},{\bf w})\mapsto {\bf v}\dot\times{\bf w}$ will be called
the \emph{   cross product} on $\F$.
The map $j:\F\to{\mathfrak g}$ is a Lie algebras isomorphism of $\F$ (endowed with 
the    cross product as composition law) onto the Lie algebra $\mathfrak g$, whose
composition law is the commutator of endomorphisms. Its transpose $j^T:{\mathfrak g}^*\to \F^*$,
defined by the equality
 $$\langle j^T(\xi), {\bf v}\rangle=\langle\xi,j({\bf v})\rangle\,,\quad
    \xi\in {\mathfrak g}^*\,,\ {\bf v}\in \F\,,
 $$
is therefore an isomorphism of the dual vector space ${\mathfrak g}^*$ of 
the Lie algebra $\mathfrak g$ onto the dual vector space 
$\F^*$ of $\F$.  
\par\smallskip

When $\zeta=1$, the    cross product is the well known \emph{cross product}
$({\bf v},{\bf w})\mapsto {\bf v}\times{\bf w}$ on the Euclidean oriented three-dimensional
vector space $\F$, and the map $j:\F\to{\mathfrak g}\equiv\so(3)$ is the isomorphism
of $\F$ onto the Lie algebra ${\mathfrak g}\equiv\so(3)$ of its Lie group 
of symmetries $G\equiv\SO(3)$, very often used in mechanics (see for example \cite{Souriau1969}). 
These remarkable properties of oriented Euclidean three-dimensional vector 
spaces therefore still hold for oriented pseudo-Euclidean three-dimensional vector spaces, 
the usual cross product being replaced with the    cross product defined above. 
\par\smallskip

Both when $\zeta=1$ and when $\zeta=-1$, for all $g\in G$, ${\bf v}$ and ${\bf w}\in\F$,
 $$j\bigl(g({\bf v})\bigr)=\Ad_g\bigl(j({\bf v})\bigr)\,,\quad
   g({\bf v})\dot\times g({\bf w})=g({\bf v}\dot\times{\bf w})\,.
 $$
The first above equality expresses the fact that the map $j$ is equivariant with 
respect to the natural action of $G$ on $\F$ and its adjoint action on its 
Lie algebra $\mathfrak g$. The second expresses the fact that the action of the group of symmetries 
$G$ preserves the    cross product.
\par\smallskip

We denote by $\scal$ the linear map defined by the equality
 $$\bigl\langle\scal({\bf u}),{\bf v}\bigr\rangle={\bf u}\cdot{\bf v}\quad\hbox{for all}\ 
    {\bf u}\ \hbox{and}\ {\bf v}\in\F\,,
 $$
where, in the left hand side, $\bigl\langle\scal({\bf u}),{\bf v}\bigr\rangle$
denotes the pairing by duality of $\scal({\bf u})\in\F^*$ with ${\bf v}\in\F$.
The map $\scal$ is an isomorphism of $\F$ onto its dual vector space $\F^*$, 
which satisfies, for any admissible basis $({\bf e}_x,{\bf e}_y,{\bf e}_z)$ of $\F$,
 $$\scal({\bf e}_x)= {\bm\varepsilon}_x\,,\quad
    \scal({\bf e}_y)= {\bm\varepsilon}_y\,,\quad
     \scal({\bf e}_z)= \zeta{\bm\varepsilon}_z\,,
 $$
where $({\bm\varepsilon}_x,{\bm\varepsilon}_y,{\bm\varepsilon}_z)$  is the basis
of $\F^*$ dual of the basis $({\bf e}_x,{\bf e}_y,{\bf e}_z)$ of $\F$.
\par\smallskip

The isomorphism $\scal:\F\to\F^*$ satisfies, for all $g\in G$ and ${\bf v}\in\F$,
 $$\scal(g({\bf v}))=(g^{-1})^T(\scal{\bf v})\,,
 $$
where $(g^{-1})^T:\F^*\to\F^*$ is the linear automorphism of $\F^*$ transpose of the
linear automorphism $g^{-1}$ of $\F$. This equality expresses the fact that
$\scal$ is equivariant with respect to the natural action of $G$ on $\F$ and 
its contragredient action on the left on $\F^*$,
$(g,\eta)\mapsto (g^{-1})^T(\eta)$, with $g\in G$, $\eta\in\F^*$.
\par\smallskip

Therefore the map $(j^{-1})^T\circ \scal:\F\to{\mathfrak g}^*$ is a linear isomorphism
which satisfies, for all $g\in G$ and ${\bf v}\in\F$,
 $$(j^{-1})^T\circ\scal\bigl(g({\bf v})\bigr)=\Ad^*_{g^{-1}}\bigl((j^{-1})^T\circ\scal({\bf v})\bigr)\,,
 $$
which expresses the fact that the isomorphism $(j^{-1})^T\circ\scal$ is equivariant 
with respect to the natural action of $G$ on $\F$ and its coadjoint action on the left
on ${\mathfrak g}^*$, $(g,\xi)\mapsto\Ad^*_{g^{-1}}\xi$,
with $g\in G$, $\xi\in{\mathfrak g}$.
\par\smallskip

In what follows the vector space  $\F$ will be identified either with the Lie algebra
$\mathfrak g$ by means of the isomorphism $j$, or with the dual vector space
${\mathfrak g}^*$ by means of the isomorphism $(j^{-1})^T\circ\scal$. We will write simply
$\F\equiv{\mathfrak g}$ when $\F$ is identified with ${\mathfrak g}$ and 
$\F\equiv{\mathfrak g}^*$ when it is identified with ${\mathfrak g}^*$, without writing explicitly
the isomorphism used for this identification. The natural action of $G$ on $F$ will therefore be
identified with its adjoint action on ${\mathfrak g}$ when $\F\equiv{\mathfrak g}$
and with its coadjoint action on the left 
on ${\mathfrak g}^*$ when $\F\equiv{\mathfrak g}^*$.  
\end{enonce}

\begin{enonce}[definition]{Expression of the map $j$ in terms of the Hodge star operator}
\label{subsub:ExpressionHodge}
\noindent
For any oriented $n$-dimensional real vector space endowed with a nondegenerate scalar product
with any signature, the \emph{Hodge star operator}, introduced by the British mathematician
W.~V.~D.~Hodge (1903--1975) is a linear automorphism of the vector space
$\bigwedge V=\oplus_{k=0}^n\bigwedge^k V$ wich, for each integer $k$ satisfying $0\leq k\leq n$, 
maps $\bigwedge^k V$ onto $\bigwedge^{n-k}V$, with, by convention, $\bigwedge^0V=\RR$
(see for example \cite{HodgeWikipedia} or \cite{ChoquetBruhat}, page 281). For the three-dimensional
vector space $\F$ considered here, the Hodge star operator satisfies, for
any admissible basis $({\bf e}_x,{\bf e}_y,{\bf e}_z)$ of $\F$, the following equalities:
 \begin{gather*}
 *(1)={\bf e}_x\wedge{\bf e}_y\wedge{\bf e}_z\,,\quad\hbox{and conversely}\
 *({\bf e}_x\wedge{\bf e}_y\wedge{\bf e}_z)=\zeta\,,\\
 *({\bf e}_x)={\bf e}_y\wedge{\bf e}_z\,,\quad\hbox{and conversely}\ 
 *({\bf e}_y\wedge{\bf e}_z)=\zeta{\bf e}_x\,,\\
 *({\bf e}_y)={\bf e}_z\wedge{\bf e}_x\,,\quad\hbox{and conversely}\ 
 *({\bf e}_z\wedge{\bf e}_x)=\zeta{\bf e}_y\,,\\
 *({\bf e}_z)=\zeta{\bf e}_x\wedge{\bf e}_y\,,\quad\hbox{and conversely}\ 
 *({\bf e}_x\wedge{\bf e}_y)={\bf e}_z\,.
 \end{gather*}
By using these formulae, one easily can check that the isomorphism $j:\F\to{\mathfrak g}$
is expressed in terms of the Hodge star operator as follows. For any triple
$(a,b,c)\in \RR^3$,
 $$j(a{\bf e}_x+b{\bf e}_y+c{\bf e}_z)=*(\zeta a{\bf e}_x+\zeta b{\bf e}_y+c{\bf e}_z)\,.
 $$
This result immediatly implies that the isomorphism $j$ does not depend on the choice
of the admissible basis used for its definition. When I first introduced $j$ 
when $\zeta=-1$ in \cite{Marle2019}, I was not aware of its expression in terms 
of the Hodge star operator. With a better choice of conventions for the definition 
of $j$, its expression in terms of the Hodge star operator could be made more natural.  
\end{enonce}

\begin{enonce}[definition]{The pseudo-Riemannian or Riemannian metric, 
the Lie algebra and the Lie-Poisson structures of $\bf F$}
\label{subsub:LieAlgebraAndPoissonOnF}
\noindent
Since, as explained at the end of  \ref{subsub:RemarkableIsomorphism},we have both
$\F\equiv{\mathfrak g}$ and $\F\equiv{\mathfrak g}^*$, the vector space $\F$
is endowed with a Lie algebra structure for which its identification with ${\mathfrak g}$
is a Lie algebras isomorphism, and with a Lie-Poisson structure for which its identification
with ${\mathfrak g}^*$ is a Poisson diffeomorphism. 
Let  $({\bf e}_x,{\bf e}_y,{\bf e}_z)$ be an admisssible basis of $\F$, and let
$x$, $y$ and $z$ be the coordinate functions on $\F$ in this admissible basis.
\par\smallskip

As seen in \ref{subsub:RemarkableIsomorphism}, the composition law of the Lie algebra 
structure of $\F$ is the    cross product 
$({\bf v},{\bf w})\mapsto{\bf v}\dot\times{\bf w}$. The non-zero brackets of ordered
pairs of elements of the considered admissible basis are
 $${\bf e}_x\dot\times{\bf e}_y=-{\bf e}_y\dot\times{\bf e}_x=\zeta{\bf e}_z\,,\ 
   {\bf e}_y\dot\times{\bf e}_z=-{\bf e}_z\dot\times{\bf e}_y={\bf e}_x\,,\ 
   {\bf e}_z\dot\times{\bf e}_x=-{\bf e}_x\dot\times{\bf e}_z={\bf e}_y\,.
 $$ 
For the Lie-Poisson structure of $\F$, the non-zero brackets of ordered pairs 
of coordinate functions are
 $$\{x,y\}=-\{y,x\}=z\,,\ \{y,z\}=-\{z,y\}=\zeta x\,,\ 
    \{z,x\}=-\{x,z\}=\zeta y\,,
 $$
and the expression of the Poisson bivector $\Lambda_{\F}$, in these coordinates, is
 $$\Lambda_{\F}(x,y,z)=z\frac{\partial}{\partial x}\wedge\frac{\partial}{\partial y}
                 +\zeta x\frac{\partial}{\partial y}\wedge\frac{\partial}{\partial z}
                 +\zeta y\frac{\partial}{\partial z}\wedge\frac{\partial}{\partial x}\,.  
 $$
Still in the coordinates functions considered, the expression of the pseudo-Riemannian 
or Riemannian metric on $\F$ determined by its scalar product is
 $$\d {s_{\F}}^2(x,y,z)=\d x^2+\d y^2+\zeta\d z^2\,.
 $$
\end{enonce}

\begin{enonce}[definition]{The coadjoint orbits of $G$ as submanifolds of $\F$}
\label{subsub:CoadjointOrbitsInF}
\noindent
Since $\F\equiv{\mathfrak g}^*$, the coadjoint orbits of $G$ can be considered as
submanifolds of $\F$. So considered they are the connected 
submanifolds of $\F$ defined as $\{{\bf v}\in \F| {\bf v}\cdot{\bf v}=\hbox{Constant}\}$,
with any possible $\hbox{Constant}\in\RR$. In other words, with the coordinate 
functions $x$, $y$ and $z$ in an admissible basis $({\bf e}_x,{\bf e}_y,{\bf e}_z)$ 
of $\F$, coadjoint orbits are connected submanifolds of $\F$ determined 
by an equation $x^2+y^2+\zeta z^2=\hbox{Constant}$,
for some $\hbox{Constant}\in\RR$. The singleton $\{0\}$, whose unique element is 
the origin of $\F$, is a zero-dimensional coadjoint orbit. All other coadjoint 
orbits are two-dimensional.
\par\smallskip 

Let $\mathcal O$ be any two-dimensional coadjoint orbit. On suitably chosen open subsets
of $\mathcal O$, one can use as coordinates two of the three coordinate functions
$x$, $y$ and $z$ associated with an admissible basis $({\bf e}_x,{\bf e}_y,{\bf e}_z)$ 
of $\F$, the third coordinate being on the chosen subset a smooth function of the
other two coordinates. Another possible choice of coordinates, which seems the most convenient,
is made of the third coordinate $z$ and the angular coordinate $\varphi$, defined by the equalities
 $$x=\sqrt{x^2+y^2}\cos\varphi\,,\quad y=\sqrt{x^2+y^2}\sin\varphi\,.
 $$
The symplectic form $\omega_{\mathcal O}$ of 
the coadjoint orbit $\mathcal O$ admits the four equivalent expressions in terms of these coordinates:
 $$\omega_{\mathcal O}=\frac{1}{z(x,y)}\d x\wedge\d y=\frac{\zeta}{x(y,z)}\d y\wedge\d z
           =\frac{\zeta}{y(z,x)}\d z\wedge\d x=\zeta\d\varphi\wedge\d z\,.
 $$
Each of the first three expressions of $\omega_{\mathcal O}$ is valid on open subsets of $\mathcal O$ on which
the coordinate considered as  a smooth function of the other two coordinates is non-zero:
$z(x,y)$ for the first, $x(y,z)$ for the second and  $y(z,x)$ for the third expression.
The fourth expression is valid on the dense open subset of $\mathcal O$ on which the angular coordinate
$\varphi$ can be locally defined, \emph{i.e.}, on the complementary subset of the set of points in $\mathcal O$ where both 
$x=0$ and $y=0$. When $\zeta=1$ this occurs only at two points of each two-dimensional coadjoint
orbit. When $\zeta =-1$, it occurs nowhere on some two-dimensional coadjoint orbits (the one sheeted hyperboloids 
denoted below by $H_R$ and the light cones with their apex removed denoted below by $C^+$ and $C^-$), 
and at a single point for other coadjoint orbits (the pseudo-spheres denoted below by $P_R^+$ and $P_R^-$).  
For this reason the coordinate system made of $z$ and $\varphi$
is the most convenient for the determination of    Gibbs states. With these coordinates, 
the expression of the Liouville measure $\lambda_{\omega_{\mathcal O}}$ is
 $$\lambda_{\omega_{\mathcal O}}(\d {\bf v})=\d z\,\d\varphi\,,\quad {\bf v}\in{\mathcal O}\ \hbox{with coordinates}\ (z,\varphi)\,.
 $$

When $\zeta=1$, all two-dimensional coadjoint orbits are spheres centered on the origin $0$ of $\F$.
Their radius can be any real $R>0$. We will denote by $S_R$ the sphere of radius $R$ centered on $0$.
It should be observed that the symplectic form on the coadjoint orbit 
$\omega_{S_R}$ \emph{is not} the area form on this sphere, 
since it is proportional to $R$, not to $R^2$. The area form of $S_R$ is $R\omega_{S_R}$.  
\par\smallskip

When $\zeta=-1$, there are three kinds of two-dimensional coadjoint orbits, described below.

\begin{itemize}

\item{} The orbits, denoted by $P_R^+$ and $P_R^-$,  whose respective equation is
 $$z=\sqrt{R^2+x^2 +y^2}\ \hbox{for}\ P_R^+\ \hbox{and}\  
   z=-\sqrt{R^2+x^2 +y^2}\ \hbox{for}\ P_R^-\,,\quad\hbox{with}\ R>0\,.
 $$
They are called \emph{pseudo-spheres} of radius $R$. Each one is a sheet of a two-sheeted two-dimensional
hyperboloid with the $z$ axis as revolution axis. They are said to be \emph{space-like} submanifolds of $\F$,
since all their tangent vectors are space-like vectors. 

\item{} The orbits, denoted by $H_R$, defined by the equation
 $$x^2+y^2=z^2+R^2\,,\quad\hbox{wit}\ R>0\,.
 $$
Each of these orbits is a single-shetted hyperboloid with the $z$ axis as revolution axis.
The tangent space at any point to such an orbit is a two-dimensional Minkowski vector space.

\item{} The two orbits, denoted by $C^+$ and $C^-$, defined respectively by
 $$z^2=x^2+y^2\ \hbox{and}\ z>0\,,\qquad z^2=x^2+y^2\ \hbox{and}\ z<0\,.
 $$
They are the cones in $\F$ (without their apex, the origin $0$ of $\F$), made of light-like vectors
directed, respectively, towards the future and towards the past.
\end{itemize}
\end{enonce}

\subsection{Gibbs states on some symplectic manifolds}
\label{subsec:4.2} 
In this subsection assumptions and notations are those of \ref{subsec:4.1}.

\begin{enonce}[definition]{   Gibbs states on two-dimensional spheres}
\label{subsub:4.2.1}
We assume here that $\zeta =1$. The Lie group $G$ is therefore isomorphic to 
$\SO(3)$ and its Lie algebra
$\mathfrak g$ is isomorphic to $\so(3)$. Let us consider the sphere $S_R$ of radius
$R>0$ centered on the origin $0$ of the vector space $\F\equiv{\mathfrak g}^*$ (identified 
with the dual ${\mathfrak g}^*$ of the Lie algebra $\mathfrak g$, as explained at the end of
\ref{subsub:RemarkableIsomorphism}). This sphere is a 
coadjoint orbit, and the moment map of the Hamiltonian action of $G$ on it
is its canonical injection into $\F\equiv{\mathfrak g}^*$. Let ${\bm \beta}\in{\mathfrak g}\equiv\F$.
Let us choose an admissible basis $({\bf e}_x,{\bf e}_y,{\bf e}_z)$ of $\F$ such that
${\bf e}_z$ and $\bm\beta$ are parallel and directed in the same direction. 
Therefore we have ${\bm\beta}=\beta{\bf e}_z$, with $\beta\geq 0$.
\par\smallskip

For each ${\bf r}=x{\bf e}_x+y{\bf e}_y+z{\bf e}_z\in S_R$, we have
  $$\bigl\langle J({\bf r}),{\bm\beta}\bigr\rangle={\bf r}\cdot{\bm\beta}=\beta z\,.
  $$
As explained in \ref{subsub:CoadjointOrbitsInF}, on the dense open subset of $S_R$ complementary to the union of
the two poles $\{-R{\bf e}_z,R{\bf e}_z\}$, we can use the coordinate system $(z,\varphi)$ and write
 \begin{align*}\int_{S_R}\exp\bigl(-\langle J({\bf r}),{\bm\beta}\rangle\bigr)\lambda_{\omega_{S_R}}(\d{\bf r})
   &=\int_0^{2\pi}\left(\int_{-R}^R\exp(-\beta z)\d z\right)\d\varphi\\
   &=\begin{cases}
     4\pi R&\text{if $\beta=0$},\\
     {\displaystyle\frac{4\pi\sh(R\beta)}{\beta}}&\text{if $\beta>0$.}
      \end{cases}
 \end{align*}
Since $S_R$ is compact, the above integral is always normally convergent. The open subset
$\Omega$ of generalized temperatures is the whole Lie algebra $\mathfrak g$. The
partition function $P$ and the probability density  $\rho_{\bm\beta}$ of the
   Gibbs state indexed by $\bm\beta$ are expressed as
 \begin{align*}
  P({\bm\beta})
   &=\begin{cases}     
     {\displaystyle\frac{4\pi\sh(R\beta)}{\beta}}&\text{if $\beta>0$,}\\
     4\pi R&\text{if $\beta=0$},
      \end{cases}\\
   \rho_{\bm\beta}({\bf r})
    &=\begin{cases}
        {\displaystyle\frac{\beta\exp(-\beta z)}{4\pi\sh(R\beta)}}&\text{if $\beta>0$,}\\
        {\displaystyle\frac{1}{4\pi R}}&\text{if $\beta=0$,}
      \end{cases}\quad\hbox{with}\ {\bf r}=x{\bf e}_x+y{\bf e}_y+z{\bf e}_z\in S_R\,.
  \end{align*}  
When $\beta>0$, the thermodynamic functions mean value of $J$ and entropy are expressed as
 \begin{align*}
 E_J({\bm\beta})
 &=\frac{1-R\beta\coth(R\beta)}{\beta^2}\,{\bm\beta}\,,\\
 S({\bm\beta})
 &=1+\log\left(\frac{4\pi\sh(R\beta)}{\beta}\right)
    -R\beta\coth(R\beta)\,.
 \end{align*}
\end{enonce}

\begin{enonce}[definition]{   Gibbs states on 
two-dimensional pseudo-spheres and other $\SO(2,1)$ coadjoint orbits}
\label{subsub:4.2.2}
We assume here that $\zeta =-1$. The Lie group $G$ is therefore isomorphic to 
$\SO(2,1)$ and its Lie algebra $\mathfrak g$ is isomorphic to $\so(2,1)$.
For each  coadjoint orbit $\mathcal O$ of $G$, we must determine whether the
integral, which defines a function of the variable ${\bm\beta}\in{\mathfrak g}$, 
 $$\int_{\mathcal O}\exp\bigl(-\bigl\langle J({\bf r}),{\bm\beta}\bigr\rangle\bigr)
    \lambda_{\omega_{\mathcal O}}(\d{\bf r})\eqno(*)
 $$
is normally convergent.
\par\smallskip

Let us first assume that $\mathcal O$ is the pseudo-sphere $P_R^+$ defined in
\ref{subsub:CoadjointOrbitsInF}, for some real number $R>0$. When the vector
${\bm\beta}\in\F\equiv{\mathfrak g}$ is time-like, we choose an admissible basis
$({\bf e}_x,{\bf e}_y,{\bf e}_z)$ of $\F$ such that 
${\bf e}_z$ and ${\bm\beta}$ are parallel. We can therefore write ${\bm\beta}=\beta{\bf e}_z$,
with $\beta\in\RR$. We have now, for each ${\bf r}=x{\bf e}_x+y{\bf e}_y+z{\bf e}_z\in{\mathcal O}\subset{\F}\equiv\F^*$,
 $$\bigl\langle J({\bf r}),{\bm\beta}\bigr\rangle={\bf r}\cdot{\bm\beta}=\zeta z\beta=-z\beta\,,
 $$
since $\zeta=-1$. We can choose $(z,\varphi)$ as coordinates on the dense open 
subset of $\mathcal O$ complementary to the singleton $\{R{\bf e}_z\}$, so the above integral
$(*)$ becomes                    
 $$\int_0^{2\pi}\left(\int_R^{+\infty}\exp(\beta z)\d z\right)\d\varphi\,.
 $$
This integral is normally convergent if and only if $\beta<0$, in other words if and
only if the time-like vector $\bm\beta$ is directed towards the past.
\par\smallskip

Still with ${\mathcal O}=P_R^+$, let us now assume that the vector $\bm \beta$ is space-like.
We choose an admissible basis $({\bf e}_x,{\bf e}_y,{\bf e}_z)$ of $\F$ such that
${\bf e}_x$ and ${\bm\beta}$ are parallel. We therefore have ${\bm\beta}=\beta{\bf e}_x$,
with $\beta\in\RR$, $\beta\neq 0$. With $(z,\varphi)$ as coordinate system, the above integral
$(*)$ is expressed as
 $$\int_0^{2\pi}\left(\int_R^{+\infty}\exp\left(-\beta \cos\varphi\sqrt{z^2-R^2}\right)\d z\right)\d\varphi\,.
 $$
This integral is always divergent, as well when $\beta<0$ as when $\beta>0$, since
$-\beta\cos\varphi>0$ for many values of $\varphi$, using the fact that for $z>0$ large enough 
$\sqrt{z^2-R^2}\equiv z$.
\par\smallskip

The subset $\Omega$ of $\F\equiv{\mathfrak g}$ of generalized temperatures
contains all time-like vectors in $\F$ directed towards the past, no time-like vector 
directed towards the future and no space-like vector. Since it is open, it cannot 
contain the origin of $\F$, nor light-like vectors. Therefore $\Omega$ is exactly
the subset of $\F$ made of time-like vectors directed towards the past.
The partition function $P$ and the probability density $\rho_{\bm\beta}$ associated
to a time-like vector $\bm\beta$ directed towards the past are expressed as
 \begin{align*}
  P(\bm\beta)
   &=\frac{2\pi}{\Vert{\bm\beta}\Vert}\exp\bigl(-\Vert{\bm\beta}\Vert R\bigr)\,,
      \quad {\bm\beta}\in\F\,,\ {\bm\beta}\ \hbox{time-like directed towards the past}\,,\\
      \rho_{\bm\beta}({\bf r})
   &=\frac{\Vert{\bm\beta}\Vert\exp\Bigl(-\Vert{\bm\beta}\Vert \bigl(z({\bf r})-R\bigr)\Bigr)}
      {2\pi}\,,\quad {\bf r}\in P_R^+\,, 
 \end{align*}
where we have set
$\Vert{\bm\beta}\Vert=\sqrt{-{\bm\beta}\cdot{\bm\beta}}$,
since ${\bm\beta}\cdot{\bm\beta}<0$.
\par\smallskip

The thermodynamic functions mean value of $J$ and entropy are
 \begin{align*}
 E_J({\bm\beta})
 &=-\frac{1+R\Vert{\bm \beta}\Vert}{\Vert{\bm\beta}\Vert^2}\,{\bm\beta}\,,\\
 S({\bm\beta})
 &=1+\log\frac{2\pi}{\Vert{\bm\beta}\Vert}\,.
 \end{align*}
\par\smallskip

Similarly, one can prove that on the pseudo-sphere $P_R^-$, the open subset $\Omega$ of
generalized temperatures is the subset of $\F$ made of time-like vectors directed towards the future, and
that the probability density of    Gibbs states and the corresponding thermodynamic functions 
are given by the same formulae as those indicated above,
of course with the appropriate sign changes. 
\par\smallskip

By similar calculations, one can prove that on the other two-dimensional coadjoint 
orbits of $G$ denoted by $H_R$, $C^+$ and $C^-$, there are no    Gibbs states,
since on these orbits the subset $\Omega$ of generalized temperatures is empty.
\end{enonce}

\begin{enonce}[definition]{   Gibb states on the Poincaré disk}
\label{subsub:4.2.3} 
Assumptions and notations here are still those of \ref{subsec:4.1}, with
$\zeta =-1$. The choice of any admissible basis $({\bf e}_x,{\bf e}_y,{\bf e}_z)$
of $\F$ determines, for each $R>0$, a diffeomorphism $\psi_R$ of the pseudo-sphere $P_R^+$
onto the Poincaré disk $D_P$, subset of the complex plane $\CC$ whose elements $w$ satisfy $\Vert w\Vert<1$.  
Its expression is
 $$\psi_R({\bf r})=\frac{x+iy}{R+\sqrt{R^2+x^2+y^2}}\,,\quad
    {\bf r}=x{\bf e}_x+y{\bf e}_y+\sqrt{R^2+x^2+y^2}{\bf e}_z\in P_R^+\,.
 $$
It is composed \cite{Marle2019} of the stereographic projection of the 
pseudo-sphere $P_R^+$ on the two-dimensional vector subspace of $\F$ generated by 
${\bf e}_x$ and ${\bf e}_y$,
 $${\bf r}=x{\bf e}_x+y{\bf e}_y+\sqrt{R^2+x^2+y^2}{\bf e}_z\mapsto
   \frac{R}{R+\sqrt{R^2+x^2+y^2}}(x{\bf e}_x+y{\bf e}_y)
 $$
with the map
 $$(u{\bf e}_x+v{\bf e}_y)\mapsto w=\frac{u+iv}{R}\,.
 $$
The expression of its inverse $\psi_R^{-1}:D_p\to P_R^+$ is
 $$\psi_R^{-1}(w)=\frac{R}{1-\vert w\vert^2}\bigl(
  2(w_{\rm r}{\bf e}_x+w_{\rm im}{\bf e_y})+(1+\vert w\vert^2){\bf e}_z\bigr)\,,
 $$
where $w=w_{\rm r}+iw_{\rm im}\in D_P$, $\vert w\vert^2=w_{\rm r}^2+w_{\rm im}^2<1$. 
\par\smallskip

The pseudo-sphere $P_R^+$ is endowed both with the Riemannian metric induced by that of $\F$,
and with its symplectic form of coadjoint orbit of the Lie group $G\equiv\SO(2,1)$
($\F$ being identified with the dual vector space ${\mathfrak g}^*$ of the Lie algebra 
${\mathfrak g}\equiv\so(2,1)$). The Poincaré disk $D_P$ is therefore endowed with a Riemannian
metric $\d {s_{D_P}}^2$ and with a symplectic form $\omega_{D_P}$ for which the map
$\psi_R$ is both an isometry and a symplectomorphism. Their expressions are
 \begin{align*}
  \d {s_{D_P}}^2(w)&=\frac{4R^2}{(1-\vert w\vert^2)^2}\,\d w\d \overline w
                     =\frac{4R^2}{(1-\vert w\vert^2)^2}(\d w_{\rm r}^2+\d w_{\rm im}^2)\,,\\
    \omega_{D_P}(w)&=\frac{2iR}{(1-\vert w\vert^2)^2}\d w\wedge\d\overline w
                     =\frac{4R}{(1-\vert w\vert^2)^2}\d w_{\rm r}\wedge\d w_{\rm im}\,,
 \end{align*}
$w_{\rm r}$ and $w_{\rm im}$ being the real and the imaginary parts of $w=w_{\rm r}+iw_{\rm im}$, respectively.
In this expression, the choice of the real number $R>0$ plays the part of the choice of a unit of length
on the Poincaré disk $D_P$.
\par\smallskip

On the open dense subset of $D_P$ complementary to the singleton $\{0\}$, the 
polar coordinate $\varphi$ such that
 $$w_{\rm r}=\vert w\vert\cos\varphi\,,\quad w_{\rm im}=\vert w\vert\sin\varphi\,,
       \quad\hbox{with}\ \vert w\vert=\sqrt{w_{\rm r}^2+w_{\rm im}^2}\,,
 $$
can be locally defined. Then we can write
 $$\d w_{\rm r}\wedge\d w_{\rm im}=\vert w\vert \d\vert w\vert\wedge\d\varphi\,,
 $$
so $\omega_{D_P}(w)$ can be expressed as
 $$\omega_{D_P}(w)=\frac{4R\vert w\vert}{(1-\vert w\vert^2)^2}\d\vert w\vert\wedge\d \varphi
    =\d\left(\frac{2R}{1-\vert w\vert^2}\right)\wedge\d\varphi\,.\,.
 $$
On this open dense subset of $D_P$, the Liouville measure $\lambda_{\omega_{D_P}}$, 
expressed in terms of the local polar coordinates $(\vert w\vert, \varphi)$, is therefore
 $$\lambda_{\omega_{D_P}}(\d w)
   =\frac{4R\vert w\vert}{(1-\vert w\vert^2)^2}\d\vert w\vert\d \varphi\,.
 $$
Moreover, the Lie group $G\equiv\SO(2,1)$ acts on the symplectic manifold $(D_P,\omega_{D_P})$ 
by a Hamiltonian action for which the diffeomorphism $\psi_R$ is equivariant, 
$P_R^+$ being identified with a coadjoint orbit on which $G$ acts by its coadjoint action on the left.
The moment map $J_{D_P}$ of this Hamiltonian action of $G$ on the Poincaré disk is 
the inverse $\psi_R^{-1}$ of $\psi_R$, whose expression is indicated above.
\par\smallskip

In $P_R^+$, the open subset $\Omega$ of generalized temperatures, 
determined in \ref{subsub:4.2.2}, is the set of time-like vectors directed towards the past. 
Let $\bm\beta$ be one of its elements. Let us choose the admissible basis
$({\bf e}_x,{\bf e}_y,{\bf e}_z)$ such that ${\bm\beta}=\beta{\bf e}_z$, with $\beta\in\RR$,
$\beta<0$. On the Poincaré disk $D_P$, the probability density of the Gibbs state indexed by 
$\bm \beta$, with respect to the Liouville measure $\lambda_{\omega_{D_P}}$, 
is deduced from $\rho_{\bm\beta}({\bf r})$ by replacing $z({\bf r})$ by its 
expression in terms of $w$, deduced from the expression of $\psi_R^{-1}$~: 
 $$z({\bf r})=\frac{R(1+\vert w\vert^2)}{1-\vert w\vert^2}\,,\quad
   \hbox{therefore}\quad z({\bf r})-R=\frac{2R}{1-\vert w\vert^2}\,.
 $$

 $$\rho_{\bm\beta}({\bf r})\d z\d{\varphi}
  =\frac{2R\vert\beta\vert\vert w\vert}{\pi(1-\vert w\vert^2)^2}
    \exp\left(-\frac{2R\vert\beta\vert}{1-\vert w\vert^2}\right)
     \d\vert w\vert\d\varphi\,,\quad {\bf r}\in P_R^+\,.
 $$
The probability density of the Gibbs state indexed by $\bm \beta$ on the Poincaré disk $D_P$, with respect to the
measure $\d\vert w\vert\d\varphi$ associated to the polar coordinate system $(\vert w\vert,\varphi)$, is therefore
 $$\frac{2R\vert\beta\vert\vert w\vert}{\pi(1-\vert w\vert^2)^2}
                                           \exp\left(-\frac{2R\vert\beta\vert}{1-\vert w\vert^2}\right)\,.
 $$
Using the expression of $\lambda_{\omega_{D_P}}$ indicated above,
we see that the probability density $\rho_{\bm\beta}$ of this Gibbs state, 
with respect to the Liouville measure $\lambda_{\omega_{D_P}}$, is
 $$\rho_{\bm\beta}(w)=\frac{\vert\beta\vert}{2\pi}
                       \exp\left(-\frac{2R\vert\beta\vert}{1-\vert w\vert^2}
                        \right)\,,\quad\hbox{with}\ w\in D_P\,.
 $$
The associated    thermodynamic functions (mean value of the moment map $E_J$ and entropy $S$) are
the functions of the generalized temperature $\bm\beta$ whose expressions are given in \ref{subsub:4.2.2}.
\par\smallskip

Instead of the Lie group $G\equiv\SO(2,1)$, the Lie group $\SU(1,1)$ is very often 
used as group of symmetries of the Poincaré disk $D_P$. It is the group of
complex $2\times 2$ matrices wich can be written as
 $$A=\begin{pmatrix}
      a&b\cr
          \overline{\mathstrut b}&\overline{\mathstrut a}
     \end{pmatrix}\,,\quad
      \hbox{with}\ a\ \hbox{and}\ b\in \CC\,,\ 
       \vert a\vert^2-\vert b\vert^2=a\overline{\mathstrut a}-b\overline{\mathstrut b}=1\,.
        \eqno(*)
 $$ 
This group acts on the Poincaré disk $D_P$ by \emph{Möbius transformations}, 
so called in honour of the German mathematician August Ferdinand Möbius
(1790--1868). We recall that the Möbius transformation determined by a complex 
$2\times 2$ matrix $\displaystyle A=\begin{pmatrix}
      a&b\cr
          c&d
     \end{pmatrix}$, with $a$, $b$, $c$ and $d\in \CC$ satisfying $ad-bc\neq 0$,
is the map $U_A:\widehat\CC\to \widehat\CC$,
with $\widehat\CC=\CC\cup\{\infty\}$, 
 $$U_A(w)=\begin{cases}{\displaystyle \frac{aw+b}{cw+d}}&\text{if $w\in \CC$ and $cw+d\neq 0$},\cr
                        {\displaystyle\infty}&\text{if $w\in \CC$ and $cw+d=0$},\cr
                        {\displaystyle \frac{a}{c}}&\text{if $w=\infty$ and $c\neq 0$},\cr
                        {\displaystyle\infty}&\text{if $w=\infty$ and $c=0$}.
          \end{cases}
 $$ 
The Möbius transformations $U_A$ and $U_{A'}$ determined by the two matrices $A$ 
and $A'$ are equal if and only if $A'=\lambda A$ for some $\lambda\in\CC$, 
$\lambda\neq 0$. When $A$ and $A'\in\SU(1,1)$, 
$U_{A'}=U_A$ if and only if $A'=\pm A$. The Möbius transformation $U_A$, determined by $A\in\SU(1,1)$, 
restricted to the Poincaré disk $D_P$, is a diffeomorphism of $D_P$ onto itself, 
and the map $\SU(1,1)\times D_P\to D_P$ so defined
is a holomorphic left action of $\SU(1,1)$ on $D_P$. There exists a surjective
Lie groups homomorphism $\Phi$ of $\SU(1,1)$ onto $\SO(2,1)$ whose kernel is the discrete group
$\{1,-1\}$ (where $1$ stands for the unit $2\times2$ matrix and $-1$ for the opposite matrix). 
For each complex $2\times2$ matrix $A\in\SU(1,1)$, expressed as indicated by the formulae $(*)$
above, $\Phi(A)$ is the real $3\times3$ matrix (see for example \cite{Marle2019})
$$\Phi(A)=\begin{pmatrix}
   {\displaystyle \frac{a^2+\overline{\mathstrut  a}^2+(b^2+\overline{\mathstrut  b}^2)}{2}} & 
   -{\displaystyle \frac{a^2-\overline{\mathstrut  a}^2-(b^2-\overline{\mathstrut  b}^2)}{2 i}} &
   -(ab+\overline{\mathstrut  a}\overline{\mathstrut  b})\cr
   {\displaystyle \frac{a^2-\overline{\mathstrut  a}^2+(b^2-\overline{\mathstrut  b}^2)}{2 i}} &
   {\displaystyle \frac{a^2+\overline{\mathstrut  a}^2-(b^2+\overline{\mathstrut  b}^2)}{2}} &
    -{\displaystyle \frac{ab-\overline{\mathstrut  a}\overline{\mathstrut  b}}{ i}} \cr
    -(a\overline{\mathstrut  b}+\overline{\mathstrut  a} b) &
    {\displaystyle \frac{-(\overline{\mathstrut  a} b-a \overline{\mathstrut  b})}{i}} &
    (a\overline{\mathstrut  a}+b\overline{\mathstrut  b})
   \end{pmatrix}\,.\eqno(**)
$$ 
The Lie algebras of the Lie groups $\SU(1,1)$ and $\SO(2,1)$ are therefore isomorphic. 
They can both be identified with the vector space $\F$, as well as their dual vector spaces.
The action of $\SU(1,1)$ on the Poincaré disk $D_P$ by Möbius transformations
can therefore be identified with the Hamiltonian action of $G\equiv\SO(2,1)$ discussed above, 
and admits $J_{D_P}$ as moment map.
\end{enonce}

\begin{enonce}[definition]{Gibbs states on the Poincaré half-plane}
\label{subsub:4.2.4}
The Möbius transformation $U_M$ determined by the matrix 
$\displaystyle M=\begin{pmatrix}-i&i\cr 1&1\end{pmatrix}$, restricted to the 
Poincaré disk $D_P$, is the map
 $$w\mapsto\xi=U_M(w)=\frac{i(-w+1)}{w+1}\,,\quad w\in D_P=\{w\in\CC\,|\,\vert w\vert<1\}\,.
 $$
Its image is the half plane 
 $$\Pi_P=\{\xi=\xi_{\rm r}+i\xi_{\rm im}\in\CC\,|\,\xi_{\rm im}>0\}\,,
 $$ 
where $\xi_{\rm r}$ and $\xi_{\rm im}$ are respectively the real and the imaginary parts
of the complex number $\xi$. Endowed with the Riemannian metric and the symplectic form
for which $U_M:D_P\to \Pi_P$ is both an isometry and a symplectomorphism, $\Pi_P$
is called the \emph{Poincaré half-plane}. The expressions of its Riemannian metric
$\d{s_{\Pi_P}}^2$ and of its symplectic form $\omega_{\Pi_P}$ are
 \begin{align*}
  \d{s_{\Pi_P}}^2(\xi)&=\frac{R^2}{\xi_{\rm im}^2}(\d \xi_{\rm r}^2+\d \xi_{\rm im}^2)\,,\\
   \omega_{\Pi_P}(\xi)&=\frac{R}{\xi_{\rm im}^2}\d\xi_{\rm r}\wedge\d\xi_{\rm im}
                       =\d\left(\frac{R}{\xi_{\rm im}}\right)\wedge\d\xi_{\rm r}\,.
 \end{align*}
The Liouville measure $\lambda_{\omega_{\Pi_P}}$ is therefore
 $$\lambda_{\omega_{\Pi_P}}(\d\xi)=\frac{R}{\xi_{\rm im}^2}\d\xi_{\rm r}\d\xi_{\rm im}\,.
 $$
A matrix
 $$A=\begin{pmatrix}
      a&b\cr
          \overline{\mathstrut b}&\overline{\mathstrut a}
     \end{pmatrix}\in\SU(1,1) \,,\quad
      \hbox{with}\ a\ \hbox{and}\ b\in \CC\,,\ 
       \vert a\vert^2-\vert b\vert^2=a\overline{\mathstrut a}-b\overline{\mathstrut b}=1\,,
 $$  
acts on the Poincaré disk $D_P$ by the Möbius transformation $U_A$. Since the Möbius tansformation 
$U_M$ determined by the matrix $M$, restricted to $D_P$, is a diffeomorphim of $D_P$
onto the Poincaré half-plane $\Pi_P$, the corresponding action of $A$ on $\Pi_P$
is the Möbius transformation determined by the matrix
 \begin{align*}
 M A M^{-1}&=\begin{pmatrix}-i&i\cr 1&1
             \end{pmatrix}
             \begin{pmatrix}a&b\cr 
             \overline{\mathstrut  b}&\overline{\mathstrut  a}
             \end{pmatrix}
             \begin{pmatrix}i/2&1/2\cr -i/2&1/2
             \end{pmatrix}\\
           &=\begin{pmatrix}
              a+\overline{\mathstrut a}-b-\overline{\mathstrut b}&
                             -i(a-\overline{\mathstrut a}+b-\overline{\mathstrut b})\cr
                             i(a-\overline{\mathstrut a}-b+\overline{\mathstrut b})&
                             a+\overline{\mathstrut a}+b+\overline{\mathstrut b}
             \end{pmatrix}\\
           &=2\begin{pmatrix}
              a_{\rm r}-b_{\rm r}&
                             a_{\rm im}+b_{\rm im}\cr
                             -a_{\rm im}+b_{\rm im}&
                             a_{\rm r}+b_{\rm r}
             \end{pmatrix}\,.                                                
 \end{align*}
The Möbius transformations determined by $MAM^{-1}$ and by $(1/2)MAM^{-1}$ being equal,
we are led to consider the map $\Sigma$, defined on $\SU(1,1)$, taking its values in the
set of real $2\times 2$ matrices, which associates to each matrix $A\in \SU(1,1)$ the matrix
 $$\Sigma(A)=\begin{pmatrix}
              a_{\rm r}-b_{\rm r}&
                             a_{\rm im}+b_{\rm im}\cr
                             -a_{\rm im}+b_{\rm im}&
                             a_{\rm r}+b_{\rm r}
             \end{pmatrix}\,,\quad\hbox{with}\ 
          A=\begin{pmatrix}
      a&b\cr
          \overline{\mathstrut b}&\overline{\mathstrut a}
     \end{pmatrix}\in\SU(1,1) \,.
 $$  
Observing that $\det\bigl(\Sigma(A)\bigr)=\vert a\vert^2-\vert b\vert^2=1$, we see that the map
$\Sigma$ is a Lie groups isomorphism of $\SU(1,1)$ onto $\SL(2,\RR)$. The Lie group 
$\SL(2,\RR)$ therefore acts on the Poincaré half-plane $\Pi_P$ by a Hamiltonian action.
As for the action of $\SU(1,1)$ on the Poincaré disk $D_P$, the dual vector space of
the Lie algebra $\sl(2,\RR)$ can be identified with the vector space $\F$. With this identification,
the expression of the moment map of the Hamiltonian action of $\SL(2,\RR)$ on the Poincaré half-plane
$\Pi_P$ is
 $$J_{\Pi_P}(\xi)=\frac{R}{2\xi_{\rm im}}\,\bigl((1-\vert \xi\vert^2){\bf e}_x
                                          +2\xi_{\rm r}{\bf e}_y
                                          +(1+\vert\xi\vert^2){\bf e}_z\bigr)\,.
 $$
The set $\Omega$ of generalized temperatures for the Hamiltonian action of $\SL(2,\RR)$
on the Poincaré half-plane $\Pi_R$ is, as for the action of $\SU(1,1)$ on the Poincaré disk
$D_P$, the set of time-like elements in $\F\equiv\sl(2,\RR)$ directed towards the past.
Let $\bm\beta$ be one of its elements. We choose an admissible basis
$({\bf e}_x,{\bf e}_y,{\bf e}_z)$ of $\F$ such that ${\bm\beta}=\beta{\bf e}_z$, with $\beta<0$.
Proceeding as for the Poincaré disk, we see that the probability density $\rho_{\bm\beta}$ 
of the    Gibbs state on $\Pi_P$ indexed by $\bm\beta$, with respect to the measure
$\d\xi_{\rm r}\d\xi_{\rm im}$, is, when expressed with the coordinate system $(\xi_{\rm r},\xi_{\rm im})$,

The set $\Omega$ of generalized temperatures, for the Hamiltonian action $\Psi$,
is the set of time-like elements in ${\bf F}\equiv\sl(2,{\mathbb R})$ 
directed towards the past. Let $\bm\beta$ be one of its elements. We choose an admissible basis
$({\bf e}_x,{\bf e}_y,{\bf e}_z)$ of ${\bf F}$ such that ${\bm\beta}=\beta{\bf e}_z$, with $\beta<0$.
The probability density $\rho_{\bm\beta}$ 
of the Gibbs state on $\Pi_P$ indexed by $\bm\beta$, with respect to the Liouville 
measure $\lambda_{\omega_{\Pi_P}}$, is
 $$\rho_{\bm\beta}(\xi)
   =\frac{\vert\beta\vert}{2\pi}\exp\left(-\frac{R\vert\beta\vert
                                                \bigl((1+\xi_{\rm im})^2+\xi_{\rm r}^2\bigr)}
                                                 {2\xi_{\rm im}}\right)\,,
    \quad \xi=\xi_{\rm r}+i\xi_{\rm im}\in\Pi_P\,.
 $$ 
The associated thermodynamic functions (mean value of the moment map $E_J$ and entropy $S$) are
the functions of the generalized temperature $\bm\beta$ whose expressions are given in \ref{subsub:4.2.2}.
\end{enonce}

\begin{enonce}[definition]{There is no    Gibbs state on a two-dimensional symplectic vector space}
\label{subsub:4.2.5}
We consider the plane $\RR^2$ (coordinates $u$, $v$), endowed with the symplectic form $\omega=\d u\wedge\d v$.
The symplectic group $\Sp(\RR^2,\omega)$ is the group $\SL(2,\RR)$ of real $2\times2$ matrices with determinant
$1$. As seen in \ref{subsub:4.2.4}, its Lie algebra, as well as its dual vector space, can be identified with
the vector space $\F$, once an admissible basis $({\bf e}_x,{\bf e}_y,{\bf e}_z)$ of $\F$ is chosen. 
The infinitesimal generators of the action of $\SL(2,\RR)$ on $\RR^2$ are the three Hamiltonian vector fields
 \begin{align*}
  X_{\RR^2}(u,v)&=\frac{1}{2}\left(v\frac{\partial}{\partial u}+u\frac{\partial}{\partial v}\right)\,,\ \hbox{whose Hamiltonian is}\ 
           H_{X_{\RR^2}}(u,v)=\frac{u^2-v^2}{4}\,,\\ 
  Y_{\RR^2}(u,v)&=\frac{1}{2}\left(u\frac{\partial}{\partial u}-v\frac{\partial}{\partial v}\right)\,,\ \hbox{whose Hamiltonian is}\ 
           H_{Y_{\RR^2}}=-\frac{uv}{2}\,,\\ 
  Z_{\RR^2}(u,v)&=\frac{1}{2}\left(v\frac{\partial}{\partial u}-u\frac{\partial}{\partial v}\right)\,,\ \hbox{whose Hamiltonian is}\ 
           H_{Z_{\RR^2}}=-\frac{u^2+v^2}{4}\,.
 \end{align*}
The infinitesimal generators $X_{\RR^2}$, $Y_{\RR^2}$ and $Z_{\RR^2}$ are the 
images, by the action on $\RR^2$ of the Lie algebra $\sl(2,\RR)\equiv\F$, of 
${\bf e}_x$, ${\bf e}_y$ and ${\bf e}_z$, respectively. We therefore obtain a 
moment map $J_{\RR^2}:\RR^2\to \F\equiv\sl(2,\RR)^*$ of this action
by writing $\bigl\langle J_{\RR^2}(u,v),{\bf e}_x\bigr\rangle=H_{X_{\RR^2}}(u,v)$, 
           $\bigl\langle J_{\RR^2}(u,v),{\bf e}_y\bigr\rangle=H_{Y_{\RR^2}}(u,v)$,  
           $\bigl\langle J_{\RR^2}(u,v),{\bf e}_z\bigr\rangle=H_{X_{\RR^2}}(u,v)$.
So we have
 $$J_{\RR^2}(u,v)=H_{X_{\RR^2}}(u,v){\bm\varepsilon}_x
                 +H_{Y_{\RR^2}}(u,v){\bm\varepsilon}_y
                 +H_{Z_{\RR^2}}(u,v){\bm\varepsilon}_z\,,
 $$
where $({\bm\varepsilon}_x,{\bm\varepsilon}_y,{\bm\varepsilon}_z)$ is the basis of
$\F^*$ dual of the basis $({\bf e}_x,{\bf e}_y,{\bf e}_z)$ of $\F$. With the identification of
$\F$ with its dual $\F^*$ by means of the scalar product on $\F$ of  signature
$(+,+,-)$, we have ${\bm\varepsilon}_x={\bf e}_x$,
${\bm\varepsilon}_y={\bf e}_y$, ${\bm\varepsilon}_z=-{\bf e}_z$. Therefore
 \begin{align*}
  J_{\RR^2}(u,v)
  &=H_{X_{\RR^2}}(u,v){\bf e}_x+H_{Y_{\RR^2}}(u,v){\bf e}_y-H_{Z_{\RR^2}}(u,v){\bf e}_z\\
  &=\frac{u^2-v^2}{4}{\bf e}_x-\frac{uv}{2}{\bf e}_y+\frac{u^2+v^2}{4}{\bf e}_z\,.
 \end{align*}
By observing that
 $$\left(\frac{u^2-v^2}{4}\right)^2+\left(\frac{uv}{2}\right)^2-\left(\frac{u^2+v^2}{4}\right)^2=0\quad
    \hbox{and}\quad \frac{u^2+v^2}{4}\geq 0\,,
 $$
we see that the $J_{\RR^2}(\RR^2)$ is the union of two coadjoint orbits of $\SL(2,\RR)$: a zero-dimensional
orbit, the singleton $\{0\}$ (where $0$ stands for the origin of $\F$), and a two-dimensional orbit,
the cone $C^+$ of light-like elements in $\F$ directed towards the future. We have seen above 
(\ref{subsub:4.2.2}) that no    Gibbs state can exist on $C^+$. Therefore
no    Gibbs state can exist on a two-dimensional symplectic vector space, 
for the natural action of the linear symplectic group.
\end{enonce}

\begin{enonce}[definition]{The    Gibbs states and thermodynamic 
functions on an affine Euclidean and symplectic plane for the group of its displacements}
\label{subsub:4.2.6}
As in the preseding section, we consider the plane $\RR^2$ (coordinates $u$, $v$)
endowed with the symplectic form $\omega=\d u\wedge \d v$. Moreover we endow it with its usual
Euclidean metric, and consider the action of its group of displacements (rotations and translations),
denoted by ${\rm E}(2,\RR)$. In matrix notations, an element of $\RR^2$ of coordinates
$(u,v)$ is represented by the column vector
$\displaystyle\begin{pmatrix}u\cr v\cr 1\end{pmatrix}$ and an element $g_{(\varphi,x,y)}$ 
of ${\rm E}(2,\RR)$ by a matrix 
$\displaystyle\begin{pmatrix} \cos\varphi&-\sin\varphi&x\cr
                              \sin\varphi&\cos\varphi&y\cr
                                0&0&1\end{pmatrix}$
depending on the three real parameters $\varphi$, $x$ and $y$. The action of ${\rm E}(2,\RR)$
on $\RR^2$ is expressed as the product of matrices
 $$\begin{pmatrix} \cos\varphi&-\sin\varphi&x\cr
                              \sin\varphi&\cos\varphi&y\cr
                                0&0&1\end{pmatrix}
  \begin{pmatrix}u\cr v\cr 1\end{pmatrix}
  =\begin{pmatrix}u\cos\varphi-v\sin\varphi+x\cr u\sin\varphi+v\cos\varphi+ y\cr 1\end{pmatrix}\,.
 $$
We denote by $({\bf e}_r,{\bf e}_x,{\bf e}_y)$ the basis of the Lie algebra ${\mathfrak e}(2,\RR)$
whose elements, identified with the corresponding matrices, are
 $${\bf e}_r=\begin{pmatrix}0&-1&0\cr 1&0&0\cr 0&0&0\end{pmatrix}\,,\quad
   {\bf e}_x=\begin{pmatrix}0&0&1\cr 0&0&0\cr 0&0&0\end{pmatrix}\,,\quad
   {\bf e}_y=\begin{pmatrix}0&0&0\cr 0&0&1\cr 0&0&0\end{pmatrix}\,.
$$ 
The corresponding fundamental vector fields on $\RR^2$ are the Hamilonian vector 
fields, generators of the action of ${\rm E}(2,\RR)$, 
\begin{align*}
 ({\bf e}_r)_{\RR^2}(u,v)&=-v\frac{\partial}{\partial u}+u\frac{\partial}{\partial v}\,,& \hbox{whose Hamiltonian is}\ 
           H_{({\bf e}_r)_{\RR^2}}(u,v)&=\frac{u^2+v^2}{2}\,,\\ 
  ({\bf e}_x)_{\RR^2}(u,v)&=\frac{\partial}{\partial u}\,,& \hbox{whose Hamiltonian is}\ 
           H_{({\bf e}_x)_{\RR^2}}(u,v)&=-v\,,\\ 
  ({\bf e}_y)_{\RR^2}(u,v)&=\frac{\partial}{\partial v}\,,& \hbox{whose Hamiltonian is}\ 
           H_{({\bf e}_y)_{\RR^2}}(u,v)&=u\,.
\end{align*}
Proceeding as in \ref{subsub:4.2.5}, we obtain the expression of the moment map
 $$J_{\RR^2}(u,v)=\frac{u^2+v^2}{2}{\bm\varepsilon}_r-y{\bm\varepsilon}_x+x{\bm\varepsilon}_y\,,
 $$
where $({\bm\varepsilon}_r,{\bm\varepsilon}_x,{\bm\varepsilon}_y)$ is the basis
of ${\mathfrak e}(2,\RR)^*$ dual of the basis $({\bf e}_r,{\bf e}_x,{\bf e}_y)$ of   
${\mathfrak e}(2,\RR)$. An element ${\bm\beta}=\beta_r{\bf e}_r+\beta_x{\bf e}_x+\beta_y{\bf e}_y$
in ${\mathfrak e}(2,\RR)$ is a generalized temperature if, considered as a function 
of $\beta_r$, $\beta_x$ and $\beta_y$, the integral
 $$\int_{\RR^2}\exp\bigl(-\langle J(u,v),{\bm\beta}\rangle\bigr)\lambda_\omega
  =\int_{\RR^2}\exp\left(-\frac{u^2+v^2}{2}\beta_r+v\beta_x-u\beta_y\right)\d u\d v
    \eqno(*)
 $$
is normally convergent. Clearly, a necessary condition for the normal convergence 
of this integral is
 $$\beta_r>0\,.\eqno(**)
 $$
When this condition is satisfied, we can write 
 $$-\frac{u^2+v^2}{2}\beta_r+v\beta_x-u\beta_y=\frac{\beta_x^2+\beta_y^2}{2\beta_r}
    -\frac{\beta_r}{2}\left[\left(u+\frac{\beta_y}{\beta_r}\right)^2
                            +\left(v-\frac{\beta_x}{\beta_r}\right)^2\right]\,.
 $$
By using on the plane $\RR^2$, instead of $(u,v)$, the polar coordinates $(\rho,\psi)$,
determined by
 $$u'=u+\frac{\beta_y}{\beta_r}=\rho\cos\psi\,,\quad v'=v-\frac{\beta_x}{\beta_r}=\rho\sin\psi\,,
 $$
we see that when $(**)$ is satisfied, the integral $(*)$ above is normally convergent.
Condition $(**)$ is therefore both necessary and sufficient for the normal convergence of
$(*)$. The set $\Omega$ of generalized temperatures, for the action of
${\rm E}(2,\RR)$ on $(\RR^2,\omega)$, is made of elements
${\bm \beta}=\beta_r{\bf e}_r+\beta_x{\bf e}_x+\beta_y{\bf e}_y\in{\mathfrak e}(2,\RR)$ 
which satisfy Condition $(**)$ above. The expression of the partition function $P$ is then
 \begin{align*}
  P({\bm\beta})
  &=\exp\left(\frac{\beta_x^2+\beta_y^2}{2\beta_r}\right)\int_0^{2\pi}
     \left(\int_0^{+\infty}\exp\left(-\frac{\beta_r\rho^2}{2}\right)\rho\d \rho\right)\d\psi\\
  &=\frac{\pi(\beta_x^2+\beta_y^2)}{\beta_r^2}\,,\quad{\bm\beta}
   =\beta_r{\bf e}_r+\beta_x{\bf e}_x+\beta_y{\bf e}_y\in{\mathfrak e}(2,\RR)\,.                
 \end{align*}    
The expression of the probability density $\rho_{\bm\beta}$ of the    Gibbs state
indexed by ${\bm\beta}\in{\mathfrak e}(2,\RR)$, with respect to the Liouville measure $\d u\d v$, is
 $$\rho_{\bm\beta}(u,v)=\exp\left(\frac{\beta_x^2+\beta_y^2}{2\beta_r}\right)
    \exp\left(-\frac{\beta_r({u'}^2+{v'}^2)}{2}\right)\,.
$$
The expressions of the    thermodynamic functions $E_J({\bm\beta})$ 
(mean value of the moment map) and $S$ (entropy) are
 \begin{align*}
  E_J({\bm\beta})
    &= \frac{2}{\beta_r}{\bm\varepsilon}_r-\frac{2\beta_x}{\beta_x^2+\beta_y^2}{\bm\varepsilon}_x
                                           -\frac{2\beta_y}{\beta_x^2+\beta_y^2}{\bm\varepsilon}_y\,,\\
  S({\bm\beta})
    &=\log P({\bm\beta})=\log\pi+\log(\beta_x^2+\beta_y^2)-\log(\beta_r^2)\,.
 \end{align*}
\par\smallskip

The expressions of the symplectic Lie group cocycle $\theta:{\rm E}(2,\RR)\to{\mathfrak e}(2,\RR)^*$
and of the symplectic Lie algebra cocycle $\Theta:{\mathfrak e}(2,\RR)\times{\mathfrak e}(2,\RR)\to\RR$
associated to the moment map $J_{\RR^2}$ can be determined by using the formulae
 $$\theta(g)=J\circ\Phi_g-\Ad^*_{g^{-1}}\circ J\,,\quad \Theta(X,Y)=\bigl\langle T_e\theta(X),Y\bigr\rangle\,,
 $$
where $g\in{\rm E}(2,\RR)$, $X$ and $Y\in{\mathfrak e}(2,\RR)$, $\Phi_g:\RR^2\to\RR^2$ being
the affine isometry of $\RR^2$ determined by the action of $g\in{\rm E}(2,\RR)$.  
Although they are not necessary for the determination of   
Gibbs states, they are indicated below.
 \begin{gather*}
  \theta(g_{(\varphi,x,y)})
   =\frac{x^2+y^2}{2}{\bm\varepsilon}_r-y{\bm\varepsilon}_x+x{\bm\varepsilon}_y\,,\\
      \Theta(r_1{\bf e}_r+x_1{\bf e}_x+y_1{\bf e}_y,r_2{\bf e}_r+x_2{\bf e}_x+y_2{\bf e}_y)
   =x_1y_2-y_1x_2\,.
 \end{gather*}  
\end{enonce}

\begin{rema}
The fact that generalized temperatures are elements of the Lie algebra ${\mathfrak e}(2,\RR)$
whose component ${\beta}_r$ on ${\bf e}_r$ is stricly positive 
may seem surprising, since there is apparently no reason explaining why 
clockwise and counter-clockwise rotations have different properties. I believe 
that it follows from the choice of $\d u\wedge\d v$ as a symplectic 
form on the plane $\RR^2$, endowed with coordinates $u$ and $v$. This choice
automatically implies the choice of an orientation of this plane: 
the Hamiltonian vector field which admits $(u^2+v^2)/2$ 
as Hamiltonian is indeed the infinitesimal generator of counter-clockwise rotations 
around the origin. Replacing $\d u\wedge\d v$ by 
its opposite $\d v\wedge\d u$ would have as consequence the replacement of this 
vector field by its opposite, which is the infinitesimal generator of clockwise rotations.
\end{rema}

\section{Final comments and thanks}
We have given a few examples of Gibbs states for the Hamiltonian action
of a non-commutative Lie group on a symplectic manifold, even when the considered 
symplectic manifold is non-compact. However, we encoutered too several examples in which
no    Gibbs state can exist, the set of generalized temperatures being empty.
All our examples are relative to two-dimensional symplectic manifolds. It seems interesting to
look now at higher-dimensional symplectic manifolds. 
\par\smallskip

Although I studied Jean-Marie Souriau's book \cite{Souriau1969} in my youth, my 
knowledge of his works in statistical mechanics were rather superficial. I owe 
to Frédéric Barbaresco, who led me to look again at this book more in depth, my 
interest in Gibbs states.
\par\smallskip

Roger Balian and Alain Chenciner were kind enough to look at this work. Their 
numerous helpful remarks and constructive criticisms, their respective helps
to better understand the foundations of quantum statistical mechanics on one hand
and Shannon's paper \cite{Shannon1948} on the other hand, were invaluable.
\par\smallskip

I benefited of very fruitful exchanges of ideas with Géry de Saxcé during our
common work on a paper in preparation.
\par\smallskip

To all of them, and to my colleagues and friends I encountered regularly in various
seminars before the beginning of the present dark period of sanitary confinment, 
I address my warmest thanks.

\backmatter

\end{document}